\documentclass[11pt]{article}

\newcommand{\nex}{x}

\newcommand{\de}{\,\mathrm{d}}                               
\newcommand{\e}{\operatorname{e}}                               
\newcommand{\im}{\mathrm{i}}

\newcommand{\p}{\partial}

\newcommand{\real}{\mathrm{Re}\,}

\newcommand{\R}{\mathbb{R}}       
\newcommand{\C}{\mathbb{C}}       
\newcommand{\N}{\mathbb{N}}

\newcommand{\mgf}{ H}                                        
\newcommand{\elf}{ E}  
  
\newcommand{\dive}{\operatorname{div}}

\newcommand{\nor}{{\nu}} 
\newcommand{\curl}{\operatorname{curl}}

\newcommand{\oP}{\mathsf{P}}  
\newcommand{\Pper}{{\oP_{\nu}}}
\newcommand{\Ppar}{{\oP_{t}}}
\newcommand{\ED}{\R^3\setminus\overline\Omega}

\newcommand{\der}{{\rm D}}

\usepackage{xcolor}
\definecolor{capa}{RGB}{50, 90, 160}
\definecolor{delftblue}{RGB}{0,61,165}
\usepackage[colorlinks=true,
  linkcolor=delftblue,
  urlcolor=delftblue,
  citecolor=delftblue]{hyperref}
\newcommand{\Id}{{\mathsf I}}  
\newcommand{\oA}{{\mathsf A}}  
\newcommand{\oB}{{\mathsf B}}  
\newcommand{\oH}{{\mathsf H}}  
\newcommand{\oR}{{\mathsf R}}  
\newcommand{\oS}{{\mathsf S}}  
\newcommand{\oK}{{\mathsf K}}  
\newcommand{\oT}{{\mathsf T}}  
\newcommand{\oC}{{\mathsf C}}  
\newcommand{\oQ}{{\mathsf Q}}  
  
\newcommand{\oL}{{\mathsf L}}

\newcommand{\oV}{{\mathsf V}}  
  
\usepackage{enumitem}

\usepackage{multirow}
\usepackage{amsmath}
\usepackage{amsfonts}
\usepackage{amsmath}
\usepackage{amssymb}  
\usepackage{graphicx}
\usepackage{caption}
\usepackage{mathrsfs}
\usepackage{upgreek}
\usepackage{amsthm}
\usepackage{subfig}
\usepackage{booktabs}
\usepackage{authblk}
\usepackage{cite}
\usepackage{bbm}
\usepackage{url}

\usepackage[]{algorithm2e}

\newtheorem{theorem}{Theorem}[section]
\newtheorem{lemma}[theorem]{Lemma}

\newtheorem{remark}[theorem]{Remark}

\title{Maxwell à la Helmholtz: Direct boundary integral equations for 3D scattering by perfect electric conductors via Helmholtz operators}
\author[1]{Carlos P\'erez-Arancibia\thanks{\href{mailto:c.a.perezarancibia@utwente.nl}{c.a.perezarancibia@utwente.nl}}}
\author[2]{Catalin Turc\thanks{\href{mailto:catalin.turc@njit.edu}{catalin.turc@njit.edu}}}
\affil[1]{\small{Department of Applied Mathematics, University of Twente, the Netherlands}}
\affil[2]{\small{Department of Mathematical Sciences, New Jersey Institute of Technology, USA}}
\date{\today}

\begin{document}
\maketitle

\begin{abstract}
This paper is the direct-formulation companion to~\cite{burbano2025maxwell}, which developed indirect combined-field-only boundary integral equations (BIEs) for time-harmonic electromagnetic scattering by smooth perfectly electrically conducting (PEC) obstacles, relying entirely on Helmholtz boundary integral operators. Here we exploit the same equivalence between the Maxwell PEC scattering problem and a pair of vector Helmholtz boundary value problems---one for the electric field and one for the magnetic field---to derive direct BIE formulations whose unknowns are the Dirichlet and Neumann traces of the total fields, decomposed into their normal and tangential surface components. These unknowns carry direct physical meaning: in particular, the magnetic-field formulation yields the surface electric currents as part of its solution. The mixed regularity of the two field-trace components requires introducing a tailored product Hölder space, a distinctive feature absent from the indirect approach. We prove that the resulting Direct Electric and Magnetic Combined-Field-Only Integral Equations (D-ECFOIE and D-MCFOIE) are uniquely solvable at all frequencies, and introduce Calderón-type regularizations (RD-ECFOIE and RD-MCFOIE) that render them of the Fredholm second kind. We further examine the low-frequency breakdown affecting the electric-field formulation and introduce a modified equation that enforces the physical charge-conservation constraints, which restores numerical accuracy and well-conditioned linear systems for frequencies arbitrarily close to zero. Numerical experiments, performed using a high-order Nyström solver based on the Density Interpolation Method and implemented in the Julia package \texttt{Inti.jl}, validate the accuracy and robustness of the proposed formulations across a range of geometries and frequencies.
\end{abstract}

\section{Introduction}

The numerical solution of frequency-domain electromagnetic scattering problems via boundary integral equations (BIEs) is widely regarded as significantly more challenging than the analogous acoustic problem governed by the Helmholtz equation. Traditional Maxwell BIE formulations---such as the electric field integral equation (EFIE), the magnetic field integral equation (MFIE), and their combined-field counterpart (CFIE)~\cite{volakis2012integral}---involve surface-tangential vector unknowns, complex hypersingular kernels, and notorious difficulties including low-frequency breakdown and the need for specialized Calderón preconditioning strategies. While important contributions have addressed these issues---including high-order Nyström methods~\cite{Bruno2009electromagnetic,ganesh2006spectrally,garza2020boundary,hu2021chebyshev} and current-and-charge reformulations incorporating Helmholtz operators~\cite{epstein2010debye,epstein2013debye,taskinen2006current,bendali2012extension}---the gap between the practical accessibility of Helmholtz and Maxwell BIE solvers remains significant.

The companion paper~\cite{burbano2025maxwell} introduced a framework---referred to as \emph{Maxwell à la Helmholtz}---that bridges this gap by reformulating 3D electromagnetic scattering by smooth perfectly electrically conducting (PEC) obstacles entirely in terms of Helmholtz boundary integral operators. The central idea rests on two equivalence theorems establishing that the Maxwell PEC scattering problem is equivalent to a pair of vector Helmholtz boundary value problems (BVPs) for the electric and magnetic fields individually, with boundary conditions expressed through the Dirichlet and Neumann traces of the respective field components. Building on these theorems, in~\cite{burbano2025maxwell} we develop \emph{indirect} combined-field-only BIE formulations---the Electric and Magnetic Field-Only Integral Equations (ECFOIE and MCFOIE)---in which the unknowns are auxiliary surface densities appearing in a combined-field layer potential ansatz. Crucially, these densities carry no direct physical meaning. The companion paper also introduces Calderón-regularized variants (R-ECFOIE and R-MCFOIE) that render the formulations of the second kind, and addresses the low-frequency breakdown inherent in the electric-field formulation via a modified combined-field potential.

The present paper develops the \emph{direct} counterpart of this framework. Rather than representing the scattered field through an ansatz in terms of auxiliary densities, we apply Green's representation formula directly to the scattered field in the exterior domain and to the incident field in the interior, combining them to obtain integral representations in terms of the Dirichlet and Neumann traces of the \emph{total} fields. These traces carry clear physical meaning: for the electric formulation, the unknown decomposes into the normal component of the total electric field on the surface and the tangential component of its normal derivative; for the magnetic formulation, the roles are interchanged. In particular, the solution of the direct magnetic formulation yields the surface electric currents directly, which are central to many applications and can be used to reconstruct the full scattered field via the classical Stratton--Chu formulas~\cite{Stratton:1939ha}.

A related class of methods---known as \emph{field-only surface integral equations}~\cite{klaseboer2016nonsingular,sun2017robust,sun2020field,sun2020field_die}---has previously explored direct reformulations of Maxwell problems in terms of Helmholtz equations. In particular,~\cite{sun2020field} uses the same non-standard boundary condition on the normal derivative of the electric field (appearing in the equivalence theorems of~\cite{burbano2025maxwell}) to construct a direct BIE. However, that formulation suffers from spurious resonances, is not of Fredholm second kind, and its low-frequency behavior is not analyzed. Spurious resonances for the case of a sphere were partially addressed in~\cite{maltez2024combined}, but neither work treats the magnetic formulation or provides a thorough well-posedness analysis. The present paper provides a complete theoretical treatment of the direct approach that was missing in the prior literature.

Concretely, the contributions of this paper are the following:
\begin{itemize}[leftmargin=*]
\item We derive direct combined-field-only integral equations for both the electric and magnetic fields---the D-ECFOIE and D-MCFOIE---whose unknowns are projections of the Dirichlet and Neumann traces of the respective total fields onto the tangential and normal components at the surface. These unknowns live in a product function space $\mathcal{C}^{p,q}_{\nu,t}(\Gamma)$, defined in~\eqref{eq:func_space}, with mixed Hölder regularity reflecting the different smoothness of the two components, which is a distinctive feature of the direct approach not present in the indirect formulation.

\item We prove uniqueness of solutions to the D-ECFOIE and D-MCFOIE for all wavenumbers $k > 0$ (Theorems~\ref{eq:unique_E} and~\ref{eq:unique_H}). The argument proceeds by reducing the homogeneous BIEs to the uniqueness question for the underlying PEC scattering problem, formulated as a pair of vector Helmholtz boundary value problems derived in~\cite{burbano2025maxwell}.

\item We introduce Calderón-type (single-layer) regularizers for both formulations (the RD-ECFOIE and RD-MCFOIE), and establish that the resulting regularized equations are of the Fredholm second kind on $C^{0,\alpha}(\Gamma,\C^3)$, from which existence of solutions follows by the Fredholm alternative.

\item We address the potential low-frequency breakdown of the electric-field formulation as $k \downarrow 0$ by introducing a modified formulation that enforces the underlying charge-conservation constraints. We prove the unique solvability of the modified formulation for all $k > 0$ under mild conditions on the stabilization parameter~($\xi$), and demonstrate numerically that it restores accuracy and well-conditioned GMRES iteration counts for frequencies arbitrarily close to $k = 0$ in the cases severely affected by the phenomenon. The magnetic-field formulation, by contrast, is unaffected by this breakdown on simply-connected surfaces.

\item We validate the proposed formulations through numerical experiments using a high-order Nyström method based on the General-Purpose Density Interpolation Method (GP-DIM)~\cite{faria2021general}, as implemented in the open-source Julia package \texttt{Inti.jl}~\cite{IntegralEquations_Inti_2025}. The examples cover smooth surfaces of varying topology and geometry under planewave incidence, and span a wide range of frequencies including the near-zero regime.
\end{itemize}

The paper is organized as follows. Section~\ref{sec:prelim} collects notation, function spaces, and the Helmholtz boundary integral operators used throughout. Section~\ref{sec:problem_setup} formulates the PEC scattering problem and recalls from~\cite{burbano2025maxwell} the equivalence theorems between the Maxwell problem and vector Helmholtz BVPs, which form the theoretical backbone of our formulations. Sections~\ref{sec:electric_field_BIE} and~\ref{sec:magnetic_field_BIE} develop the direct electric and magnetic combined-field-only formulations, respectively, together with their Calderón-regularized variants and low-frequency-robust modifications. Finally, Section~\ref{sec:numerics} presents the numerical experiments.

\section{Preliminaries\label{sec:prelim}}
The essential definitions and notation used throughout the paper are collected in this section. 

\paragraph{Function spaces}
First, we introduce the notation for the function spaces that will be used throughout the paper. Let $X \subset \mathbb{R}^m$ and $Y \subset \mathbb{C}^n$. Given $p \in \mathbb{N}_0$ and $\alpha \in (0,1)$, the Hölder space $C^{p,\alpha}(X, Y)$ is defined as 
$$
C^{p,\alpha}(X, Y) = \left\{ F=(f_1,\cdots,f_n): X \to Y \ \middle| \ f_i \in C^p(X), \ \forall i\in\{1,\ldots,n\}, \ \text{and} \ \forall |\beta| = p, \ \der^\beta f_i \in C^{0,\alpha}(X) \right\},
$$
with the norm
$$
\|F\|_{C^{p,\alpha}(X,Y)} := \sum_{i=1}^n \left( \sum_{|\beta| \le p} \|\der^\beta f_i\|_{C(X)} + \sum_{|\beta| = p} \left[\der^\beta f_i\right]_{C^{0,\alpha}(X)} \right).
$$
Throughout, we use standard multi-index notation, and the symbol $\der$ denotes the total derivative. For the definition and properties of the standard Hölder space $C^{0,\alpha}(X)$, as well as the norm $\|\cdot\|_{C(X)}$ and the seminorm $[~\cdot~]_{C^{0,\alpha}(X)}$, we refer the reader to~\cite[Sec. 4.1]{gilbarg2001elliptic} or~\cite[Sec. 2.2]{COLTON:1983}. 

\paragraph{Domain geometry and boundary regularity}
We let $\Gamma$ denote a smooth surface that encloses an open, bounded region $\Omega \subset \mathbb{R}^3$, and assume that the exterior domain $\mathbb{R}^3 \setminus \overline{\Omega}$ is connected. The boundary $\Gamma = \partial \Omega$ admits the decomposition $\Gamma = \Gamma_1 \cup \cdots \cup \Gamma_J$ into $J$ pairwise disjoint, closed, bounded, and connected components $\Gamma_j$, for $j \in \{1, \ldots, J\}$, each assumed to be at least $C^{2,\alpha}$-smooth for some $\alpha \in (0,1)$. The outward unit normal vector field is denoted by $\nu$ and belongs to $C^{1,\alpha}(\Gamma, \mathbb{R}^3)$. This regularity guarantees that the shape operator, defined as $\mathscr{R} = \der_\Gamma \nu$, where $\der_\Gamma$ denotes the surface derivative, belongs to $C^{0,\alpha}(\Gamma, \mathbb{R}^3)$. Similarly, the mean curvature, given by $\mathscr{H} = \frac{1}{2} \dive_\Gamma \nu$ in terms of the surface divergence, lies in $C^{0,\alpha}(\Gamma)$.

Specifically, we let ${\rm x}: \mathcal N \subset \mathbb{R}^2 \to \Gamma$ be a $C^{2,\alpha}$-smooth local regular parametrization mapping an open set $\mathcal N$ onto a neighborhood of a point on $\Gamma$. The coefficients of the first and second fundamental forms are then given by
$$
E = {\rm x}_u \cdot {\rm x}_u,\quad F = {\rm x}_u \cdot {\rm x}_v,\quad G = {\rm x}_v \cdot {\rm x}_v,
$$
and
$$
L = -{\rm x}_u \cdot \nu_u, \quad M = -{\rm x}_u \cdot \nu_v = -{\rm x}_v \cdot \nu_u, \quad N = -{\rm x}_v \cdot \nu_v,
$$
respectively, where subscripts denote partial derivatives with respect to the local coordinates $(u,v)\in \mathcal N$. Letting ${\rm x}_u$ and ${\rm x}_v$ be represented as column vectors, the shape operator $\mathscr{R}$ can then be expressed in terms of these coefficients as
\begin{equation*}\label{eq:curvature_operator}
\mathscr{R} = -[{\rm x}_u \ {\rm x}_v] \begin{bmatrix} E & F \\ F & G \end{bmatrix}^{-1} \begin{bmatrix} L & M \\ M & N \end{bmatrix} \begin{bmatrix} E & F \\ F & G \end{bmatrix}^{-1} [{\rm x}_u \ {\rm x}_v]^\top,
\end{equation*}
and the mean curvature is recovered via $\mathscr{H} = \frac{1}{2} \operatorname{tr} \mathscr{R}$. As is well known~\cite{do2016differential}, both $\mathscr{R}$ and $\mathscr{H}$ are intrinsic geometric quantities of the surface, independent of the particular parametrization, and can be defined globally on all of $\Gamma$. In fact they define the following multiplication operators, which will arise naturally in our boundary integral equation formulations:
\begin{align}
\oR : C^{p,\alpha}(\Gamma, \C^3) \to C^{0,\alpha}(\Gamma, \C^3), & \quad (\oR \varphi)(x) = \mathscr{R}(x) \varphi(x),\quad\text{  and} \label{eq:R_op} \\
\oH : C^{p,\alpha}(\Gamma, \C^3) \to C^{0,\alpha}(\Gamma, \C^3), & \quad (\oH \varphi)(x) = \mathscr{H}(x) \varphi(x),\quad x\in\Gamma,\ p\in \N_0. \label{eq:H_op}
\end{align}

\paragraph{Traces and projectors}
Next, we introduce the notation for the exterior and interior traces of sufficiently regular $\mathbb{C}^n$-valued vector fields, with $n \in \{1, 3\}$, which will be used throughout the paper to impose boundary conditions and to define the decomposition of surface unknowns.
Let $U^+ \subset \mathbb{R}^3 \setminus \overline\Omega$ (resp. $U^- \subset \Omega$) be an open set such that its closure contains~$\Gamma$. 
The exterior and interior Dirichlet traces are then defined as
\begin{subequations}\begin{align}
\gamma^+: C^{p,\beta}(\overline{U^+},\C^n)\to C^{p,\beta}(\Gamma,\C^n),&\quad \gamma^+ F(\nex) = \lim_{\delta\to 0+} F(\nex+\delta\nor(\nex)),\label{eq:Dir_trace_def}\\
\gamma^-: C^{p,\beta}(\overline{U^-},\C^n)\to C^{p,\beta}(\Gamma,\C^n),&\quad \gamma^- F(\nex) = \lim_{\delta\to 0+} F(\nex-\delta\nor(\nex)),\quad \nex\in\Gamma,
\end{align}\label{eq:Dir_traces}\end{subequations}
where $(p,\beta)\in(\{0,1\}\times [0,1])\cup(\{2\}\times [0,\alpha])$, respectively. 

Similarly, we introduce the following notation for the exterior and interior Neumann traces: 
\begin{subequations}
\begin{align}
\p^+_\nu: C^{p+1,\beta}(\overline{U^+},\C^n)\to C^{p,\beta}(\Gamma,\C^n),&\quad\p^+_\nu F(\nex) := \lim_{\delta\to 0+}\der F(x+\delta\nu(x)) \nor(\nex),\label{eq:Neu_trace_def}\\
\p^-_\nu: C^{p+1,\beta}(\overline{U^-},\C^n)\to C^{p,\beta}(\Gamma,\C^n),&\quad\p^-_\nu F(\nex) := \lim_{\delta\to 0+}\der F(x-\delta\nu(x)) \nor(\nex),\quad\nex\in\Gamma,
\end{align}\label{eq:Neu_traces}\end{subequations} 
for $(p,\beta)\in(\{0,1\}\times [0,1])\cup(\{1\}\times [0,\alpha])$, respectively, 
where 
$\der: F\mapsto(\der F)_{i,j}= \frac{\p F_{i}}{\p x_j}$,  $i\in\{1,\ldots,n\}$, $ j\in\{1,2,3\},$
 denotes the total derivative (which is represented by the Jacobian matrix $(\p F_i/\p x_j)_{i,j}$). The trace operators $\gamma^\pm$ and $\partial^\pm_\nu$ are well-defined and bounded due to the assumed $C^{2,\alpha}$-regularity of the surface $\Gamma$, with $\alpha \in (0,1)$. Throughout this paper, unless the interior trace is explicitly required, we adopt the convention
$\gamma := \gamma^+$ and  $\p_\nu := \p^+_\nu$, i.e., $\gamma$ and $\p_\nu$ without a superscript always denote the exterior traces. We also use the alternative notation $\frac{\partial F}{\partial \nu} := \p^+_\nu F$ for the exterior Neumann trace of a field $F$.

It is important to note that, when $n = 3$, the above definition~\eqref{eq:Neu_traces} implies that the Neumann traces~$\p^\pm_\nu F$ are a vector field whose coordinate components correspond to the scalar Neumann traces of each component $(F)_i=F_i$ ($i \in \{1, 2, 3\}$) of the vector field $F$, i.e., $(\p^\pm_\nu F)_i=\p^\pm_\nu F_i$.

We make extensive use of the projectors on the tangent plane of $\Gamma$ and along its normal. Specifically, we define the operators:
\begin{subequations}\begin{align}
\oP_\nu:C^{p,\beta}(\Gamma,\C^3)\to C^{p,\beta}(\Gamma,\C^3),&\quad  \Pper\varphi(x)=(\nu(x)\cdot\varphi(x))\nu(x)\quad\text{and}\label{eq:P_perp_op}\\
 \oP_t:C^{p,\beta}(\Gamma,\C^3)\to C^{p,\beta}(\Gamma,\C^3),&\quad  \Ppar\varphi(x)=\varphi(x)-(\nu(x)\cdot\varphi(x))\nu(x),\quad x\in\Gamma,\label{eq:P_par_op}
\end{align}\label{eq:projectors}\end{subequations}
for $(p,\beta)\in(\{0\}\times [0,1])\cup(\{1\}\times [0,\alpha])$,  which map the vector $\varphi(x)$ onto the normal and the tangent plane at $x\in\Gamma$, respectively. Clearly, since $\nu\in C^{1,\alpha}(\Gamma,\R^3)$, both projectors are bounded.  

\paragraph{Helmholtz integral operators}
Our boundary integral equation formulations are based on Helmholtz layer potentials and boundary integral operators, which act component-wise on the three Cartesian components of a vector density, introduced below. We begin by defining the vector-valued Helmholtz layer potentials and their associated boundary integral operators. Let $\mathcal{S}: C^{0,\alpha}(\Gamma, \C^3) \to C^{2}(\R^3 \setminus \Gamma, \C^3)$ and $\mathcal{D}: C^{0,\alpha}(\Gamma, \C^3) \to C^{2}(\R^3 \setminus \Gamma, \C^3)$, with $\alpha \in (0,1)$, denote the Helmholtz single- and double-layer potentials, respectively, defined as
\begin{align}
\mathcal S[\varphi](\nex) :=&\int_\Gamma G(x,y)\varphi(y)\de s(y)\quad\text{and}\label{eq:SL_pot}\\
\mathcal D[\varphi](\nex):=&\int_{\Gamma} \frac{\p G(x,y)}{\p \nu(y)}\varphi(y)\de s(y),\quad x\in\R^3\setminus\Gamma,\label{eq:DL_pot}    
\end{align}
for vector density functions $\varphi\in C^{0,\alpha}(\Gamma,\C^3)$, where $$
G(x,y) := \frac{\e^{\im k|x-y|}}{4\pi |x-y|},\quad x\neq y,\quad k\geq0,
$$ is the free-space Green's function for the Helmholtz equation.  Here and throughout this work, the (vectorial) layer potentials~\eqref{eq:SL_pot} and~\eqref{eq:DL_pot}, along with the associated (vectorial) boundary integral operators introduced below, are to be interpreted component-wise. 

It follows from the standard jump relations for scalar Helmholtz layer potentials~\cite[Thms. 2.12 \& 2.13]{COLTON:1983} that analogous jump conditions hold for the vector-valued Helmholtz layer potentials. Specifically, consider a sufficiently regular vector density function $\varphi:\Gamma \to \C^3$. The vectorial layer potentials~\eqref{eq:SL_pot} and~\eqref{eq:DL_pot} admit extensions to the closure of both the exterior domain $\R^3 \setminus \Omega$ and the interior domain $\overline{\Omega}$, and satisfy the following limiting relations as the target point $x \in \R^3 \setminus \Gamma$ approaches the boundary $\Gamma$ from either side: 
\begin{subequations}\label{eq:jumps_LP}\begin{align}
\gamma^\pm (\mathcal S\varphi) =&~\oS\varphi,&
\gamma^\pm (\mathcal D\varphi) =&\pm\frac{1}{2}\varphi+\oK\varphi,\label{eq:jump_Dir}\\
\p_\nu^\pm (\mathcal S\varphi) =&\mp\frac{1}{2}\varphi+ \oK'\varphi,&
\p_\nu^\pm (\mathcal D\varphi) =&~\oT\varphi,\label{eq:jump_Neu}
\end{align}\end{subequations}
where $\gamma^\pm$ and $\p_\nu^\pm$  denote the exterior/interior Dirichlet and Neumann traces defined in~\eqref{eq:Dir_traces} and~\eqref{eq:Neu_traces}, respectively, and $\oS$, $\oK$, $\oK'$ and $\oT$ are the (vector) boundary integral operators: 
\begin{subequations}\begin{align}
\oS:C^{0,\alpha}(\Gamma,\C^3)\to C^{1,\alpha}(\Gamma,\C^3),&\quad \oS[\varphi](\nex):=\int_{\Gamma}G(x,y)\varphi(y)\de s(y),\label{eq:single_op}\\
\oK:C^{0,\alpha}(\Gamma,\C^3)\to C^{1,\alpha}(\Gamma,\C^3),&\quad \oK[\varphi](\nex) :=\int_{\Gamma}\frac{\p G(x,y)}{\p \nu(y)}\varphi(y)\de s(y),\label{eq:double_op}\\
\oK':C^{0,\alpha}(\Gamma,\C^3)\to C^{0,\alpha}(\Gamma,\C^3),&\quad \oK'[\varphi](\nex):= \int_{\Gamma}\frac{\p G(x,y)}{\p \nu(x)}\varphi(y)\de s(y),\label{eq:adj_op}\\
\oT:C^{1,\alpha}(\Gamma,\C^3)\to C^{0,\alpha}(\Gamma,\C^3),&\quad \oT[\varphi](\nex) :={\rm f.p.}\!\int_{\Gamma}\frac{\p^2 G(x,y)}{\p\nu(x)\p \nu(y)}\varphi(y)\de s(y),\label{eq:hyper_op}
\end{align}\label{eq:BIOs}\end{subequations}
for $\nex\in\Gamma$, which are well defined component-wise and bounded under the assumption that $\Gamma$ is $C^{2,\alpha}$-smooth. This result directly follows from the properties of the corresponding scalar operators~\cite[Thm. 3.4]{COLTON:2012}. The hypersingular integral defining $\oT$ is to be understood in the sense of the Hadamard finite-part integral~\cite{hackbusch1995integral}. 

The following theorem, whose proof follows directly from \cite[Thm. 3.2]{COLTON:2012} and \cite[Thm. 2.31]{COLTON:1983}, together with the componentwise (diagonal) nature of the operators defined in~\eqref{eq:BIOs}, summarizes  key mapping properties of these operators that will be used throughout the remainder of the paper. We emphasize that the subscript~0 in $\oT_0$, $\oS_0$, $\oK_0$, and $\oK'_0$ indicates the hypersingular, single-layer, double-layer, and adjoint double-layer operators, respectively, each defined as in~\eqref{eq:BIOs} for the case $k = 0$.
\begin{theorem}
The operators $\oS$, $\oK'$, $\oK$, and $\oT - \oT_0$ are compact on $C^{0,\alpha}(\Gamma, \mathbb{C}^3)$. 
\end{theorem}

\section{Problem setup}\label{sec:problem_setup}
This paper addresses the problem of time-harmonic electromagnetic scattering arising when an incident electromagnetic field $(E^i, H^i)$ impinges on the surface $\Gamma$ of a PEC obstacle, embedded in an unbounded, homogeneous, isotropic medium characterized by dielectric and magnetic constants $\epsilon > 0$ and $\mu > 0$, respectively. 

The incident electromagnetic field $(E^i, H^i)$, where each field belongs to $C^{1,\alpha}(U, \mathbb{C}^3)$, is assumed to be defined in an open set $U \subset \mathbb{R}^3$ containing~$\Omega$, where it satisfies the time-harmonic Maxwell equations:
\begin{equation}\label{eq:maxwell_inc}
\curl E^i - \im\omega \mu H^i =0 \quad \text{and} \quad \curl H^i + \im\omega \epsilon E^i = 0 \quad\text{in}\quad U,
\end{equation}
with $\omega > 0$ denoting the angular frequency.

The scattered electric and magnetic fields, denoted by $\elf^s$ and $\mgf^s$, respectively, are generated by the interaction of the incident field with the PEC boundary. These scattered fields satisfy the time-harmonic Maxwell equations:
\begin{subequations}\begin{equation}\label{eq:maxwell}
\curl \elf^s - \im\omega \mu\mgf^s = 0 \quad \text{and} \quad \curl \mgf^s + \im\omega \epsilon \elf^s = 0 \quad \text{in} \quad \ED,
\end{equation}
with the electric field subject to the PEC boundary condition~\cite{kirsch2016mathematical}: 
\begin{equation}\label{eq:PEC_BC}
\oP_t\gamma\elf^s =  -\oP_t \gamma\elf^i \quad\text{on}\quad\Gamma,
\end{equation}
which is written using the tangential projector~\eqref{eq:P_par_op}.  The scattered field is further required to satisfy the Silver--Müller radiation condition at infinity, given by   either
\begin{equation}\label{eq:rad_cond_sm}\begin{split}
\lim _{|\nex| \rightarrow \infty} |\nex| \left( \curl\elf^s(\nex) \times \frac{\nex}{|\nex|} - \im k\elf^s(\nex) \right) = 0\quad\text{or}\quad
\lim _{|\nex| \rightarrow \infty} |\nex| \left( \curl H^s(\nex) \times \frac{\nex}{|\nex|} - \im kH^s(\nex) \right) = 0,
\end{split}\end{equation}\label{eq:scatt_problem}\end{subequations}
where the limit holds uniformly with respect to all directions $\nex/|\nex|$ and where $k:=\omega\sqrt{\epsilon\mu}>0$ is the wavenumber.  The time dependence $\e^{-\im \omega t}$ is assumed throughout the paper.

The following theorem, whose proof can be found in~\cite[Corollary 5.14]{burbano2025maxwell}, establishes the well-posedness of the scattering problem~\eqref{eq:scatt_problem} in H\"older spaces suitable for our analysis:
\begin{theorem}\label{eq:reg_solution}
Under the stated assumptions on the surface $\Gamma$ and the incident fields $E^i$ and $H^i$, the PEC scattering problem~\eqref{eq:scatt_problem} admits a unique solution
$E^s, H^s \in C^2(\mathbb{R}^3 \setminus \overline{\Omega}, \mathbb{C}^3) \cap C^{1,\alpha}(\mathbb{R}^3 \setminus \Omega, \mathbb{C}^3)$ for all $\omega>0$.
\end{theorem}

We conclude this section by recalling two theorems from~\cite{burbano2025maxwell} that are central to our formulations. They establish that the electromagnetic scattering problem~\eqref{eq:scatt_problem} is equivalent to a pair of vector Helmholtz boundary value problems---one for the electric field and one for the magnetic field---in which the boundary conditions are expressed solely in terms of the Dirichlet and Neumann traces of the respective fields. This is the key feature that enables the direct BIE formulations developed in the following sections.

\begin{theorem} \label{eq:equiv} The following statements are equivalent for fields $E^s,H^s\in C^2(\ED,\C^3)\cap C^{1,\alpha}(\R^3\setminus\Omega,\C^3)$:
\begin{enumerate}
\item\label{thm:part_1}  $E^s$ and $H^s$ solve the PEC scattering problem~\eqref{eq:scatt_problem} with incident fields $E^i,H^i\in C^{1,\alpha}(U,\C^3)$ satisfying~\eqref{eq:maxwell_inc}.  
\item\label{thm:part_2}  $H^s = (\im\omega\mu)^{-1}\curl E^s$ and $E^s$ satisfies:
\begin{subequations}\begin{align}\label{eq:vec_Helm_eqn}
\Delta E^s + k^2E^s=&~  0\quad\text { in }\  \ED,\\
\Ppar(\gamma E^s)=&-\Ppar(\gamma E^i),\label{eq:E_bc_t}\\
\Pper(\p_\nu E^s+2\mathscr H \gamma E^s) =& -\Pper(\p_\nu E^i+2\mathscr H \gamma E^i),\label{eq:E_bc_n}\\
\lim _{|\nex| \rightarrow \infty}|x|\left\{(\der E^s)\frac{\nex}{|x|}  - \im k E^s\right\} =&~0\ \ \text{uniformly in}\ \ \frac{\nex}{|\nex|}.\label{eq:sommerfeld_condition}
\end{align}\label{eq:elf_equiv}\end{subequations}
\item $E^s = -(\im\omega\epsilon)^{-1}\curl H^s$ and $H^s$ satisfies 
\begin{subequations}
\begin{align}
\Delta H^s + k^2H^s=&~ 0\quad\text { in }\  \ED,\label{eq:vec_Helm_eqn_H}\\
\Pper(\gamma H^s)=&~-\Pper(\gamma H^i)\label{eq:H_bc_n}\\
\Ppar(\p_\nu H^s+\mathscr R \gamma H^s) =&~ -\Ppar(\p_\nu H^i+\mathscr R \gamma H^i)\label{eq:H_bc_t}\\
\lim _{|\nex| \rightarrow \infty}|x|\left\{(\der H^s)\frac{\nex}{|x|}  - \im k H^s\right\} =&~0\ \ \text{uniformly in}\ \ \frac{\nex}{|\nex|}.\label{eq:sommerfeld_condition_H}
\end{align}
\label{eq:mgf_equiv}\end{subequations}
\end{enumerate}
\end{theorem}
\begin{proof} See Theorems 3.4 and 3.5 in~\cite{burbano2025maxwell}.
\end{proof}

\section{Direct electric combined-field-only formulation}\label{sec:electric_field_BIE}
In this section, we derive a direct boundary integral equation formulation for the electric field scattering problem~\eqref{eq:elf_equiv}, and study its existence and uniqueness.

Under the assumption that the incident fields $E^i$ and $H^i$ satisfy~\eqref{eq:maxwell_inc}, we have
$\curl \curl E^i - k^2 E^i = 0$ in $\Omega\subset U.$
Using the vector identity $\curl \curl E^i = \nabla \dive E^i - \Delta E^i$ and the fact that $\dive E^i = 0$ in $\Omega$, it follows that
$\Delta E^i + k^2 E^i = 0$ in $\Omega$. Applying Green's representation theorem (component-wise) to $E^i$ in $\Omega$ then yields
$$
0= \mathcal D[\gamma E^i](x) - \mathcal S[\p_\nu E^i](x),\quad x\in\ED
$$

Similarly, leveraging the fact that $E^s$ satisfies the vector Helmholtz equation~\eqref{eq:vec_Helm_eqn} and the component-wise Sommerfeld radiation condition~\eqref{eq:sommerfeld_condition}, we apply Green's representation formula for the Helmholtz equation to the scattered field $E^s$ in $\ED$, and obtain:
$$
E^s(x) = \mathcal D[\gamma E^s](x) - \mathcal S[\p_\nu E^s](x),\quad x \in \ED.
$$

By adding these two integral representations, we arrive at the following expression:
\begin{equation}\label{eq:rep_for}
E^s(x) =  \mathcal D[\gamma E](x) - \mathcal S[\p_\nu E](x), \qquad x \in \mathbb{R}^3 \setminus \overline{\Omega},
\end{equation}
where we introduce the following notation for the traces of the total field $E:=E^s+E^i$ defined in $(\ED)\cap U$:
$$\gamma E := \gamma E^s + \gamma E^i\in C^{1,\alpha}(\Gamma,\C^3)\quad\text{and}\quad \p_\nu E := \p_\nu E^s + \p_\nu E^i\in C^{0,\alpha}(\Gamma,\C^3).$$

Expressing the boundary conditions~\eqref{eq:E_bc_n} and~\eqref{eq:E_bc_t} in terms of the traces of the total field, we obtain:
$$
\gamma E = \oP_\nu\gamma E\quad\text{and}\quad \p_\nu E = \oP_t\p_\nu E - 2\oH \oP_\nu\gamma E,
$$
where $\oP_\nu$ and $\oP_t$ denote the normal and tangential projections, respectively, as defined in~\eqref{eq:projectors}, and $\oH$ denotes the mean curvature multiplication operator~\eqref{eq:H_op}.

Upon substituting these identities into Green's representation formula~\eqref{eq:rep_for}, we obtain:
\begin{equation}\label{eq:rep_for_E}
E^s(x) = \mathcal D\left[\oP_\nu\gamma E\right](x) - \mathcal S\left[\oP_t\p_\nu E - 2\oH \oP_\nu\gamma E\right](x), \qquad x \in \mathbb{R}^3 \setminus \overline{\Omega}.
\end{equation}

Taking the exterior Dirichlet and Neumann traces on both sides of~\eqref{eq:rep_for_E} yields:
\begin{subequations}\begin{align}
-\gamma E^i =& -\tfrac{1}{2} \oP_\nu\gamma E + (\oK + 2 \oS \oH)\oP_\nu\gamma E - \oS \oP_t\p_\nu E,\label{eq:dir_trace_E}\\
-\p_\nu E^i =&-\tfrac12(\oP_t\p_\nu E - 2\oH\oP_\nu\gamma E)+ (\oT+2\oK'\oH)\oP_\nu\gamma E - \oK'\oP_t\p_\nu E.\label{eq:nue_trace_E}
\end{align}\end{subequations}

By taking the linear combination~\eqref{eq:nue_trace_E}$+\im\eta$\eqref{eq:dir_trace_E} of the Neumann and Dirichlet trace equations with $\eta\in \R\setminus\{0\}$, we arrive at
\begin{equation*}\begin{aligned}
-\tfrac12\{\oP_t\p_\nu E+ (\im\eta\Id- 2\oH)\oP_\nu\gamma E\}+ \{\oT+\im\eta \oK+2(\oK'+\im\eta\oS)\oH\}\oP_\nu\gamma E - (\oK'+\im\eta\oS)\oP_t\p_\nu E=\hspace{1cm}\\-\p_\nu E^i-\im\eta\gamma E^i.
 \end{aligned}\label{eq:DECFOIE}
\end{equation*}

At first glance, the relation above may appear to define an underdetermined BIE, since both the Dirichlet and Neumann traces, $\gamma E$ and $\p_\nu E$, are involved. However, this is not the case, as we now show. Before proceeding, we introduce the function space in which we will seek our direct BIE solutions. For $p, q \in \{0,1\}$ and $\alpha \in (0,1)$---where $\alpha$ corresponds to the $C^{2,\alpha}$-regularity of the surface $\Gamma$---we define the space
\begin{equation}\label{eq:func_space}
\mathcal{C}^{p,q}_{\nu,t}(\Gamma) := \left\{ \varphi:\Gamma\to\C^3\mid \oP_\nu \varphi \in C^{p,\alpha}(\Gamma,\C^3),\ \oP_t\varphi \in C^{q,\alpha}(\Gamma,\C^3) \right\},
\end{equation}
equipped with the norm
\begin{equation}\label{eq:func_space_norm}
\|\varphi\|_{\mathcal C^{p,q}_{\nu,t}(\Gamma)} := \|\oP_\nu \varphi\|_{C^{p,\alpha}(\Gamma,\C^3)} + \|\oP_t\varphi\|_{C^{q,\alpha}(\Gamma,\C^3)}.
\end{equation}

\begin{lemma}\label{lem:banach_space}
$\mathcal{C}^{p,q}_{\nu,t}(\Gamma)$ is a Banach space under the norm~\eqref{eq:func_space_norm}.
\end{lemma}
\begin{proof}
That~\eqref{eq:func_space_norm} defines a norm follows readily: if $\|\varphi\|_{\mathcal C^{p,q}_{\nu,t}(\Gamma)}=0$ then $\oP_\nu\varphi=0$ and $\oP_t\varphi=0$, so $\varphi = \oP_\nu\varphi+\oP_t\varphi = 0$; homogeneity and the triangle inequality are inherited from those of the Hölder norms.

For completeness, let $\{\varphi_n\}\subset \mathcal{C}^{p,q}_{\nu,t}(\Gamma)$ be a Cauchy sequence. By definition~\eqref{eq:func_space_norm}, $\{\oP_\nu\varphi_n\}$ is Cauchy in $C^{p,\alpha}(\Gamma,\C^3)$ and $\{\oP_t\varphi_n\}$ is Cauchy in $C^{q,\alpha}(\Gamma,\C^3)$. Since both Hölder spaces are Banach, there exist limits $f\in C^{p,\alpha}(\Gamma,\C^3)$ and $g\in C^{q,\alpha}(\Gamma,\C^3)$ such that $\oP_\nu\varphi_n\to f$ and $\oP_t\varphi_n\to g$. Define $\varphi:=f+g$. Since $\oP_\nu+\oP_t=\Id$, we have $\oP_\nu\varphi = \oP_\nu f + \oP_\nu g = f\in C^{p,\alpha}(\Gamma,\C^3)$ and $\oP_t\varphi = \oP_t f + \oP_t g = g \in C^{q,\alpha}(\Gamma,\C^3)$, where we used that $f = \lim_n \oP_\nu\varphi_n$ lies in the range of $\oP_\nu$ and $g = \lim_n \oP_t\varphi_n$ lies in the range of $\oP_t$, which are closed subspaces of $C^{p,\alpha}(\Gamma,\C^3)$ and $C^{q,\alpha}(\Gamma,\C^3)$, respectively. Hence $\varphi\in\mathcal{C}^{p,q}_{\nu,t}(\Gamma)$, and $\|\varphi_n - \varphi\|_{\mathcal{C}^{p,q}_{\nu,t}(\Gamma)} = \|\oP_\nu\varphi_n - f\|_{C^{p,\alpha}} + \|\oP_t\varphi_n - g\|_{C^{q,\alpha}} \to 0$.
\end{proof}

Our direct electric combined-field-only integral equation (D-ECFOIE) is then given by:
\begin{subequations}\label{eq:electric_BIE}\begin{equation}\begin{aligned}
-\tfrac12\{\varphi_t+ (\im\eta\Id- 2\oH)\varphi_\nu\}+ \{\oT+\im\eta \oK+2(\oK'+\im\eta\oS)\oH\}\varphi_\nu - (\oK'+\im\eta\oS)\varphi_t=f,
 \end{aligned}
 \end{equation}
where the unknown density is 
\begin{equation}\label{eq:unknown_density}
\varphi := \varphi_t + \varphi_\nu \in \mathcal{C}_{\nu,t}^{1,0}(\Gamma),\quad \varphi_t:=\oP_t(\p_\nu E),\quad  \varphi_\nu:=\oP_\nu(\gamma E)\end{equation} 
and the right-hand-side is given by 
\begin{equation}\label{eq:f_datum}
 f:=-\p_\nu E^i-\im\eta\gamma E^i\in C^{0,\alpha}(\Gamma,\C^3).
\end{equation}\end{subequations}
The scattered electric field can then be retrieved from the BIE solution via the representation formula~\eqref{eq:rep_for_E}.

Clearly, by leveraging the surface projection operators, the equation above can be recast in the form
\begin{equation}\label{eq:L_e_op}
\oL_e\varphi = f,
\end{equation}
where $\oL_e: \mathcal C^{1,0}_{\nu,t}(\Gamma)\to C^{0,\alpha}(\Gamma,\C^3)$
is the operator defined by
\begin{equation}\label{eq:L_electric}
\oL_e := -\tfrac12\{\oP_t+ (\im\eta\Id- 2\oH)\oP_\nu\}+ \{\oT+\im\eta \oK+2(\oK'+\im\eta\oS)\oH\}\oP_\nu - (\oK'+\im\eta\oS)\oP_t.
\end{equation}

We are now ready to establish the uniqueness of solutions to the above BIE:
\begin{theorem}\label{eq:unique_E} The D-ECFOIE given by~\eqref{eq:electric_BIE} admits at most one solution $\varphi \in \mathcal{C}_{\nu,t}^{1,0}(\Gamma)$ for all wavenumbers $k > 0$ and all coupling parameters $\eta \in \mathbb{R} \setminus \{0\}$.
\end{theorem}
\begin{proof}
Suppose there exists a non-trivial $\varphi \in \mathcal{C}^{1,0}_{\nu,t}(\Gamma)$ satisfying the homogeneous BIE:
\begin{equation}\label{eq:homo}\begin{aligned}
-\tfrac12\{\varphi_t+ (\im\eta\Id- 2\oH)\varphi_\nu\}+ \{\oT+\im\eta \oK+2(\oK'+\im\eta\oS)\oH\}\varphi_\nu - (\oK'+\im\eta\oS)\varphi_t=0.
 \end{aligned}
\end{equation}

Let $\varphi_\nu := \oP_\nu \varphi$ and $\varphi_t := \oP_t \varphi$, and define the layer potential: 
\begin{equation}\label{eq:layer_potential_uniqueness}
U(x) :=  \mathcal D[\varphi_\nu](x) - \mathcal S[\varphi_t - 2\oH\varphi_\nu](x),\quad x\in\R^3\setminus\Gamma.
\end{equation}

By taking the traces of the layer potential~\eqref{eq:layer_potential_uniqueness} and applying the jump relations~\eqref{eq:jumps_LP}, it follows that
\begin{equation}\label{eq:traces_pot}
\gamma^\pm U = \pm\tfrac12\varphi_\nu + (\oK+2\oS\oH)\varphi_\nu - \oS\varphi_t\quad\text{and}\quad
\partial_\nu^\pm U = \pm\tfrac12(\varphi_t-2\oH\varphi_\nu) + (\oT+2\oK'\oH)\varphi_\nu - \oK'\varphi_t.
\end{equation}

Therefore, it is clear from this expression that~\eqref{eq:homo} corresponds to the Robin boundary condition:
$\p^-_\nu U+\im\eta\gamma^-U =0.$
Since the restriction $U|_{\Omega} \in C^2(\Omega, \mathbb{C}^3) \cap C^{1,\alpha}(\overline{\Omega}, \mathbb{C}^3)$ of  the layer potential $U$ to the interior domain $\Omega$ satisfies the Helmholtz equation $\Delta U|_{\Omega} + k^2 U|_{\Omega} = 0$ in $\Omega$,
it follows from the uniqueness of the interior Robin boundary value problem that $U|_{\Omega}=0$ and hence $U = 0$ in $\overline{\Omega}$. Consequently, both interior traces $\gamma^-U$ and $\p_\nu^- U$ vanish. Using the jump relations from~\eqref{eq:traces_pot}, we then obtain
\begin{equation}\label{eq:rel_traces}
\gamma^+U-\gamma^-U=\gamma^+U=\varphi_\nu\quad\text{and}\quad \p_\nu^+U-\p_\nu^-U=\p_\nu^+U=\varphi_t-2\oH\varphi_\nu.
\end{equation}

From the expressions for the exterior traces obtained above, we observe that both $\varphi_\nu$ and $\varphi_t$ vanish if and only if the exterior Dirichlet and Neumann traces of the potential $U$ vanish. To show that the exterior traces vanish, we note that from~\eqref{eq:rel_traces} it follows that
$$
\oP_t\gamma^+ U=0\quad\text{and}\quad \oP_\nu(\p^+_\nu U+2\oH\gamma^+ U) = 0.
$$

Then, given that in addition, $U$ in~\eqref{eq:layer_potential_uniqueness} satisfies the Helmholtz equation~\eqref{eq:vec_Helm_eqn} as well as the radiation condition~\eqref{eq:sommerfeld_condition}, by the equivalence between the Maxwell and vector Helmholtz formulations established in Theorem~\ref{eq:equiv}, we conclude that the restriction $U|_{\ED} \in C^{2}(\ED, \mathbb{C}^3)\cap C^{1,\alpha}(\R^3\setminus\Omega,\C^3)$ of the layer potential~$U$ to the exterior domain solves the PEC scattering problem~\eqref{eq:scatt_problem} with homogeneous boundary data, i.e., $\oP_t\gamma E^i=0$. Since this problem admits a unique solution by Theorem~\ref{eq:reg_solution}, it follows that $U = 0$ in $\mathbb{R}^3 \setminus \Omega$.

This leads to a contradiction, because the boundary data $\varphi = \oP_\nu\gamma^+ U + \oP_t \p_\nu^+ U$ must then vanish identically, implying that the original density $\varphi$ is zero. This completes the proof of uniqueness.
\end{proof}

We now present the following lemma, which is used to  establish the existence of solutions to the integral equation~\eqref{eq:DECFOIE} by means of the Fredholm alternative.
\begin{lemma} \label{lem:comm_bounded} The  operators
\begin{subequations}\begin{align}
\oC_\nu :C^{0,\alpha}(\Gamma,\C^3)\to C^{1,\alpha}(\Gamma,\C^3),&& \oC_\nu \varphi:=&~\nu S_0(\nu\cdot\varphi)-\oS_0\oP_\nu\varphi,\label{eq:C_nu}\\
\oC_t:C^{0,\alpha}(\Gamma,\C^3)\to C^{1,\alpha}(\Gamma,\C^3),&& \oC_t\varphi :=&~\oP_t\oS_0\varphi-\oS_0\oP_t\varphi,\label{eq:C_t}
\end{align}\label{eq:commutators}\end{subequations}
where $S_0: C^{0,\alpha}(\Gamma) \to C^{1,\alpha}(\Gamma)$ denotes the scalar Laplace single-layer operator, admit the integral representations:
\begin{align*}
\oC_\nu[\varphi](x) =&~\int_{\Gamma}\frac{\nu(y)\cdot\varphi(y)}{4\pi|x-y|}\left\{\nu(x)-\nu(y)\right\}\de s(y),\\
\oC_t[\varphi](x) =&~-\int_{\Gamma}\frac{1}{4\pi|x-y|}\cdot\{\varphi(y)\cdot\nu(x)\nu(x)-\varphi(y)\cdot\nu(y)\nu(y)\}\de s(y),\quad x\in\Gamma,
\end{align*}
and are compact.
\end{lemma}
\begin{proof} 
  The well-definedness of the operators~\eqref{eq:commutators} and their integral representations follow directly by substituting the integral kernel of the Laplace single-layer operator~\eqref{eq:single_op} (with $k=0$) into the definitions~\eqref{eq:C_nu} and~\eqref{eq:C_t}.

In order to prove that $\oC_\nu$ in~\eqref{eq:C_nu} is compact, we examine its Cartesian components, namely $\oC_{\nu,j}$ for $j \in \{1,2,3\}$. It hence suffices to show that the operators $\delta \oC_{\nu,j}: C^{0,\alpha}(\Gamma,\C^3) \to C^{0,\alpha}(\Gamma,\C^3)$, defined by $\delta \oC_{\nu,j}\varphi := \nabla_\Gamma(\oC_{\nu,j} \varphi),$ are compact.  

Let $\varphi\in C^{0,\alpha}(\Gamma,\C^3)$. Applying the surface gradient to the integral representation of $\oC_{\nu,j}$ and using the well-known property that the tangential derivative of the single-layer operator $S_0$ can be computed by differentiation under the integral sign (see, e.g.~\cite[Thm. 2.17]{COLTON:1983}), we obtain
\begin{equation}\begin{split}
\delta \oC_{\nu,j}[\varphi](x) = \int_{\Gamma} \frac{\varphi(y)\cdot\nu(y)}{4\pi|x - y|^3} \left\{(x-y)-\nu(x)\cdot(x-y)\nu(x)\right\} \left\{ \nu_j(x) - \nu_j(y) \right\} \de s(y)+\hspace{2cm}\\
 \nabla_\Gamma \nu_j(x)\int_{\Gamma} \frac{\varphi(y)\cdot\nu(y)}{4\pi|x - y|}   \de s(y), \quad x \in \Gamma,
\end{split}\label{eq:comp_diff} \end{equation}
where all the surface integrals are to be understood in the improper sense. Note that they are well defined by virtue of the fact that each of the integrands is weakly singular. 

In view of the mapping properties of the Laplace single-layer operator $S_0$ and the fact that $\nabla_\Gamma \nu_j \in C^{0,\alpha}(\Gamma, \mathbb{C}^3)$, we readily conclude that the last term in~\eqref{eq:comp_diff} defines a compact operator on $C^{0,\alpha}(\Gamma, \mathbb{C}^3)$. On the other hand, the mapping properties of the operator defined by the first integral in~\eqref{eq:comp_diff} are governed by the smoothness of the kernels defining each of the operator's Cartesian components, which are given by
$$
\kappa_{i,j}(x,y) := \frac{\left\{(x_i-y_i)-\nu(x)\cdot(x-y)\nu_i(x)\right\} \left\{ \nu_j(x) - \nu_j(y) \right\}}{|x-y|^3},\quad \Gamma\ni x\neq y\in\Gamma,\quad i,j\in\{1,2,3\}.
$$
where $\nu_j \in C^{1,\alpha}(\Gamma)$ denote the components of the unit normal vector.
Clearly, since the Cartesian components of the unit normal $\nu_j$ are Lipschitz continuous, there exists a constant $L > 0$ such that 
\begin{equation}\label{eq:cond_1}
|\kappa_{i,j}(x,y)| \leq L |x - y|^{-1}
\end{equation}
for all $x, y \in \Gamma$, $x\neq y$.
Moreover, for all $x,y,z\in\Gamma$, we have
\begin{equation*}\begin{split}
\kappa_{i,j}(x,y)-\kappa_{i,j}(z,y) =&~ \frac{(x_i-z_i)\left\{ \nu_j(x) - \nu_j(y) \right\}}{|x-y|^3}-\frac{\left\{\nu(x)\cdot(x-z)\nu_i(x)\right\} \left\{ \nu_j(x) - \nu_j(y) \right\}}{|x-y|^3}\\
&-\frac{\left\{(\nu(x)-\nu(z))\cdot(z-y)\nu_i(x)\right\} \left\{ \nu_j(x) - \nu_j(y) \right\}}{|x-y|^3}\\
&-\frac{\left\{\nu(z)\cdot(z-y)(\nu_i(x)-\nu_i(z)\right\} \left\{ \nu_j(x) - \nu_j(y) \right\}}{|x-y|^3}\\
&+\frac{\left\{(z_i-y_i)-\nu(z)\cdot(z-y)\nu_i(z)\right\} \left\{ \nu_j(x) - \nu_j(z) \right\}}{|x-y|^3}\\
&+\left\{(z_i-y_i)-\nu(z)\cdot(z-y)\nu_i(z)\right\} \left\{ \nu_j(z) - \nu_j(y) \right\}\left\{\frac{1}{|x-y|^3}-\frac{1}{|z-y|^3}\right\}.
\end{split}\end{equation*}
Taking absolute value  and applying the triangle inequality, we arrive at
\begin{equation}\begin{split}
|\kappa_{i,j}(x,y)-\kappa_{i,j}(z,y)| \leq&~ 2\frac{|x-z|}{|x-y|^2}+2\frac{|x-z||z-y| }{|x-y|^2}+2\frac{|z-y||x - z| }{|x-y|^3}\\
&+2|z-y|^2\left|\frac{1}{|x-y|^3}-\frac{1}{|z-y|^3}\right|
\end{split}\label{eq:interm_2}\end{equation}
Then, under the assumption that $2|x - z| \leq |x - y|$, we have (see \cite[Eq.~(72)]{burbano2025maxwell} with $m = 3$):
$$
|z-y|\leq \frac32|x-y|\quad\text{and}\quad \left|\frac{1}{|x-y|^3}-\frac{1}{|z-y|^3}\right|\leq 48\frac{|x-z|}{|x-y|^{4}}.
$$
Using these inequalities in~\eqref{eq:interm_2} we obtain  
\begin{equation}\label{eq:cond_2}\begin{split}
|\kappa_{i,j}(x,y)-\kappa_{i,j}(z,y)| \leq&~ 2\frac{|x-z|}{|x-y|^2}+3\frac{|x-z| }{|x-y|}+3\frac{|x - z| }{|x-y|^2}+216\frac{|x-z|}{|x-y|^{2}}.
\end{split}\end{equation}
It then follows that the kernels $\kappa_{i,j}$ satisfy all the conditions in~\cite[Thm.2.7]{COLTON:1983}, so we can conclude that the first integral in~\eqref{eq:comp_diff} defines a bounded operator from $C^{0,\alpha}(\Gamma, \C^3)$ to $C^{0,\beta}(\Gamma, \C^3)$ for all $\beta \in (\alpha, 1]$. Then, using the compact embedding of $C^{0,\beta}(\Gamma, \mathbb{C}^3)$ into $C^{0,\alpha}(\Gamma, \mathbb{C}^3)$ for any $\beta > \alpha$~\cite[Thm. 2.5]{COLTON:1983}, we conclude that each operator $\delta \oC_{\nu,j}$ for $j \in \{1,2,3\}$ is compact on $C^{0,\alpha}(\Gamma, \mathbb{C}^3)$. Hence, $\oC_\nu$ is compact as an operator from $C^{0,\alpha}(\Gamma, \mathbb{C}^3)$ to $C^{1,\alpha}(\Gamma, \mathbb{C}^3)$.

The compactness of the operator $\oC_t: C^{0,\alpha}(\Gamma,\mathbb{C}^3) \to C^{1,\alpha}(\Gamma,\mathbb{C}^3)$ can be established by noting that it can be expressed as
$$
\oC_t[\varphi](x)=-\nu(x)\int_{\Gamma}\frac{\varphi(y)\cdot\{\nu(x)-\nu(y)\}}{4\pi|x-y|} \de s(y)-\int_{\Gamma}\frac{\varphi(y)\cdot\nu(y)}{4\pi|x-y|}\{\nu(x)-\nu(y) \}\de s(y)
$$
where the second term is exactly $\oC_\nu\varphi$, and the first term can be analyzed analogously by exploiting the smoothness of the unit normal vector.

This completes the proof. 
\end{proof}

We are now ready to state and prove the well-posedness of~\eqref{eq:electric_BIE}.
\begin{theorem}\label{thm:well_poss_E}  
The D-ECFOIE~\eqref{eq:electric_BIE} is uniquely solvable in $\mathcal{C}^{1,0}_{\nu,t}(\Gamma)$ for all wavenumbers $k>0$ and coupling parameters $\eta\in\R\setminus\{0\}$.
\end{theorem}  
\begin{proof} We begin the proof by introducing the following regularizing operator:
\begin{equation}\label{eq:precond_nu}
\oR_\nu:\mathcal C^{1,0}_{\nu,t}(\Gamma)\to C^{1,\alpha}(\Gamma,\C^3), \qquad \oR_\nu\widetilde\varphi := -2\oP_t\widetilde\varphi-4\nu S_0(\nu\cdot\widetilde\varphi),
\end{equation}
which, in view of the invertibility of the Laplace single-layer operator $S_0:C^{0,\alpha}(\Gamma)\to C^{1,\alpha}(\Gamma)$, is invertible~\cite[Thm. 7.40]{kress2012linear}. Indeed, its inverse
$\oR^{-1}_\nu:C^{1,\alpha}(\Gamma,\C^3)\to \mathcal C^{1,0}_{\nu,t}(\Gamma)$
is well defined, bounded and given by the formula
$\oR^{-1}_\nu \varphi = -\frac12\oP_t \varphi -\frac14 \nu S_0^{-1}(\nu\cdot\varphi).$

We then perform the substitution $\varphi = \oR_\nu \widetilde\varphi$ into~\eqref{eq:electric_BIE}, where $\widetilde\varphi\in C^{0,\alpha}(\Gamma,\C^3)$.  This substitution yields
\begin{eqnarray*}
\oP_t\widetilde\varphi+ 2(\im\eta\Id- 2\oH)\nu S_0(\nu\cdot\widetilde\varphi)-4 \{\oT+\im\eta\oK+2(\oK'+\im\eta\oS)\oH\}[\nu S_0(\nu\cdot\widetilde\varphi)] +2(\oK'+\im\eta\oS)\oP_t\widetilde\varphi=f.
\end{eqnarray*}

Next, using the identity
$\nu S_0(\nu\cdot\widetilde\varphi)= \oC_\nu\widetilde\varphi+\oS_0\oP_\nu\widetilde\varphi,$
which follows from the definition of the operator $\oC_\nu$ in~\eqref{eq:C_nu}, we can rewrite the equation above as
\begin{eqnarray}\label{eq:par_split}
\oP_t\widetilde\varphi+ 2(\im\eta- 2\oH)\nu\oS_0(\nu\cdot\widetilde \varphi)-4 \{\oT+\im\eta\oK+2(\oK'+\im\eta\oS)\oH\}\{\oC_\nu\widetilde\varphi+\oS_0\oP_\nu\widetilde\varphi\} +2 (\oK'+\im\eta\oS)\oP_t\widetilde\varphi=f.
\end{eqnarray}

From Calderón's identity
\begin{equation}\label{eq:calderon}
\oT_0 \oS_0 = {\oK'_0}^2 - \tfrac{1}{4} \Id,
\end{equation}
for the operators defined with $k = 0$, we obtain
\begin{equation}\label{eq:follow_calderon}
\oT\oS_0 = \oT_0\oS_0 + (\oT - \oT_0)\oS_0 = -\tfrac{1}{4}\Id + {\oK'_0}^2 + (\oT - \oT_0)\oS_0.
\end{equation}
Therefore, substituting this expression into~\eqref{eq:par_split}, the equation can be recast as
\begin{subequations}\label{eq:refor}\begin{equation}
\widetilde\varphi + \oA \widetilde\varphi = f
\end{equation}
where $\oA := \sum_{j=1}^4\oA_j$ and the operators  $\oA_j:C^{0,\alpha}(\Gamma,\C^3)\to C^{0,\alpha}(\Gamma,\C^3)$ for $j\in\{1,2,3,4\}$ are defined by:
\begin{align}
\oA_1\varphi: =&~ 2(\im\eta\Id- 2\oH)\nu S_0(\nu\cdot\varphi),\\
\oA_2\varphi:=&-4\{{\oK'_0}^2+(\oT-\oT_0)\oS_0+\im\eta\oK\oS_0+2(\oK'+\im\eta\oS)\oH\oS_0\}\oP_\nu\varphi,\\
\oA_3\varphi:=&-4\{\oT+\im\eta\oK+2(\oK'+\im\eta\oS)\oH\}\oC_\nu\varphi ,\\
\oA_4\varphi:=&~2 (\oK'+\im\eta\oS)\oP_t\varphi.
\end{align}\end{subequations}

As it turns out, each of these operators, and therefore the full operator $\oA$, is compact on $C^{0,\alpha}(\Gamma,\C^3)$. In particular, $\oA_1$ is compact, as it is the composition of the bounded operator $2(\im\eta\Id - 2\oH)$ with the compact operator $\oC_\nu + \oS_0 \oP_\nu$. Regarding $\oA_2$, the operator $\oT - \oT_0$, along with all the other operators appearing within the curly brackets in its definition, are compact on $C^{0,\alpha}(\Gamma,\C^3)$. Since $\oP_\nu$ is bounded in that space, their composition with $\oP_\nu$ remains compact. Similarly, $\oA_3$ is compact, as it is the composition of the compact operator $\oC_\nu : C^{0,\alpha}(\Gamma,\C^3) \to C^{1,\alpha}(\Gamma,\C^3)$ (see Lemma~\ref{lem:comm_bounded}) with the bounded operator $4\{\oT + \im\eta\oK + 2(\oK' + \im\eta\oS)\oH\} : C^{1,\alpha}(\Gamma,\C^3) \to C^{0,\alpha}(\Gamma,\C^3)$.
Finally, $\oA_4$ is compact on $C^{0,\alpha}(\Gamma,\C^3)$, as it is the composition of the compact operator $2(\oK' + \im\eta\oS)$ with the bounded operator $\oP_t$.

Now, to establish the uniqueness of solutions to the regularized BIE~\eqref{eq:refor}, suppose that it admits two distinct solutions $\widetilde\varphi_j \in C^{0,\alpha}(\Gamma,\C^3)$ for $j \in \{1,2\}$. Then, by the injectivity of the regularizer, the corresponding densities $\varphi_j := \oR_\nu \widetilde\varphi_j \in \mathcal{C}^{1,0}_{\nu,t}(\Gamma)$ are both solutions to the homogeneous boundary integral equation~\eqref{eq:homo}, and satisfy $\varphi_1 \neq \varphi_2$. This, however, contradicts Theorem~\ref{eq:unique_E}. We therefore conclude that~\eqref{eq:refor} admits at most one solution $\widetilde\varphi \in C^{0,\alpha}(\Gamma,\C^3)$. Finally, by the Fredholm alternative, since~\eqref{eq:refor} is of the form ``invertible plus compact", the existence of such a solution follows.

Finally, from the unique solution $\widetilde\varphi \in C^{0,\alpha}(\Gamma,\C^3)$ of the regularized BIE~\eqref{eq:refor}, we recover the unique solution of the original BIE~\eqref{eq:electric_BIE} by applying the regularizer, that is, $\varphi = \oR_\nu \widetilde\varphi \in \mathcal{C}^{1,0}_{\nu,t}(\Gamma)$. This completes the proof.

\end{proof}

\subsection*{Low-frequency breakdown}
As shown in~\cite{burbano2025maxwell}, the low-frequency breakdown of the indirect (R-)ECFOIE formulations is rooted in the lack of uniqueness of the $k=0$ limit of problem~\eqref{eq:elf_equiv}: in the zero-frequency limit, the homogeneous exterior boundary value problem for the vector Laplacian admits non-trivial solutions unless the surface charge integrals are prescribed for each boundary component~\cite{werner1963perfect_reflection} (see also~\cite[Remark~4.2]{burbano2025maxwell}). The same issue affects the direct formulations of the present paper: both the D-ECFOIE~\eqref{eq:electric_BIE} and the RD-ECFOIE~\eqref{eq:refor} suffer from low-frequency breakdown as $k\downarrow 0$, meaning that the accuracy of the numerical solutions and the conditioning of the linear systems arising from the discretization of the BIEs degrade as $k\downarrow0$.  The remedy, as in the indirect case, is to enforce the charge conditions satisfied by any physical solution of~\eqref{eq:elf_equiv}, namely
$$
\int_{\Gamma_j}\nu\cdot \gamma E^s\,\de s = 0,\qquad j\in\{1,\ldots,J\},
$$
which are precisely the constraints that restore uniqueness in the zero-frequency limit~\cite{werner1963perfect_reflection}. Well-conditioning can then be enforced by augmenting $\oL_e$ with a low-rank perturbation that penalizes violations of these charge conditions.

To this end, we introduce the following  bounded linear functionals:
\begin{equation}\label{eq:charge_funcs}
\ell_j : C^{0,\alpha}(\Gamma,\C^3) \to \C, \qquad \ell_j(\varphi) := \int_{\Gamma_j} \nu \cdot \varphi ,\mathrm{d}s, \quad j \in \{1, \ldots, J\},
\end{equation}
along with the vector fields
\begin{equation}\label{eq:normal_vec_fields}
\phi_j :=\oL_e\nu_j\in C^{0,\alpha}(\Gamma,\C^3),\qquad\text{where}\qquad \nu_j(x):=\begin{cases}\nu(x),&x\in\Gamma_j,\\0,&x\in\Gamma\setminus\Gamma_j,\end{cases}\qquad j\in\{1,\ldots,J\}.
\end{equation}


We then consider the following modification of the D-ECFOIE:
\begin{equation}\label{eq:mod_DECFOIE}
\Big( \oL_e + \xi \sum_{j=1}^J \phi_j \ell_j \Big) \varphi = f,
\end{equation}
where $\oL_e$ is the operator defined in~\eqref{eq:L_electric}, and the unknown density $\varphi = \oP_\nu \gamma E + \oP_t \p_\nu E \in \mathcal{C}^{1,0}_{\nu,t}(\Gamma)$.

Similarly, we propose a corresponding modification of the regularized formulation~\eqref{eq:refor}, given by:
\begin{equation}\label{eq:mod_reg_DECFOIE}
\Big( \Id + \oA + \xi \sum_{j=1}^J \phi_j \ell_j \oR_\nu \Big) \widetilde{\varphi} = f,
\end{equation}
for $\widetilde{\varphi} \in C^{0,\alpha}(\Gamma,\C^3)$, where the regularizing operator $\oR_\nu$ is defined in~\eqref{eq:precond_nu}.

Note that we have used the identity $\oL_e \oR_\nu = \Id + \oA$, as established in the proof of Theorem~\ref{thm:well_poss_E}. Also observe that the modified formulation~\eqref{eq:mod_reg_DECFOIE} remains a second-kind BIE, as the added term corresponds to a finite-rank operator, which is compact.


We now present the following lemma, which establishes the existence of solutions to the modified equations:
\begin{lemma} Let $\varphi \in \mathcal{C}^{1,0}_{\nu,t}(\Gamma)$ be the unique solution to the D-ECFOIE~\eqref{eq:electric_BIE}. Then, for any $\xi \in \C$, $\varphi$~also satisfies the modified equation~\eqref{eq:mod_DECFOIE}. Similarly, if $\widetilde{\varphi} \in C^{0,\alpha}(\Gamma, \C^3)$ is the unique solution to the RD-ECFOIE~\eqref{eq:refor}, then  it also solves the modified regularized equation~\eqref{eq:mod_reg_DECFOIE} for all $\xi\in\C$.
\end{lemma}
\begin{proof} 
Let $\varphi = \oP_\nu \gamma E + \oP_t \p_\nu E \in \mathcal{C}^{1,0}_{\nu,t}(\Gamma)$ be the unique solution to the D-ECFOIE~\eqref{eq:electric_BIE} with datum  defined in~\eqref{eq:f_datum}. Then, for each $j \in \{1, \ldots, J\}$,
\begin{equation*}
\ell_j(\varphi) = \int_{\Gamma_j} \nu \cdot \gamma E \, \mathrm{d}s = \int_{\Gamma_j} \nu \cdot \gamma E^s \, \mathrm{d}s + \int_{\Gamma_j} \nu \cdot \gamma E^i \, \mathrm{d}s.
\end{equation*}

Since the incident field $E^i$ is assumed to satisfy Maxwell's equations in $U$ and therefore also in~$\Omega$, we have $\dive E^i = 0$ in $\Omega$. Applying the divergence theorem over the subdomain $\Omega_j \subset \Omega$ yields
\begin{equation*}
\ell_j(\gamma E^i) = \int_{\Gamma_j} \nu \cdot \gamma E^i \, \mathrm{d}s = \int_{\Omega_j} \dive E^i \, \mathrm{d}x = 0.
\end{equation*}

To analyze the contribution of the scattered field, we first observe that the traces
$\gamma E^s = \oP_\nu(\varphi) - \gamma E^i$ and $\p_\nu E^s = \oP_t \varphi - 2\mathscr{H} \oP_\nu \varphi - \p_\nu E^i$
allow us to recover the scattered field $E^s$ throughout the exterior domain $\ED$ via Green's representation formula. We can then make use of the series expansion of the scalar field $E^s(x) \cdot x$ in the region $\mathbb{R}^3 \setminus \overline{B_R(0)}$, where $B_R(0)$ is a ball large enough to contain the domain $\Omega$. As shown in~\cite{faria2021general}, this field admits the expansion
\begin{equation*}
E^s(x) \cdot x = \sum_{l=1}^\infty \sum_{m=-l}^l c_l^m \, h_l^{(1)}(k|x|) Y_l^m\left(\frac{x}{|x|}\right), \qquad |x| \ge R,
\end{equation*}
where $\{Y_l^m\}_{m=-l}^{l}$ are spherical harmonics and $\{h_l^{(1)}\}_{l\in\N_0}$ are spherical Hankel functions of the first kind. Since the spherical harmonics satisfy
\begin{equation*}
\int_{\partial B_1(0)} Y_l^m(\hat{x}) \, \mathrm{d}s(\hat{x}) = 0 \quad \text{for all } l \in \mathbb{N}, \; -l \le m \le l,
\end{equation*}
it follows that
\begin{equation*}
\frac1{R}\int_{\partial B_R(0)} E^s(x) \cdot x \de s(x)=\int_{\partial B_R(0)} E^s(x) \cdot \frac{x}{|x|} \de s(x) = 0.
\end{equation*}

Now, define the scalar field $u_j := \dive E^s$ in $\mathbb{R}^3\setminus\overline{\Omega_j}$, which vanishes in the exterior domain $\mathbb{R}^3 \setminus \overline{\Omega}$, and is extended by zero to all of $\Omega$ except $\Omega_j$, i.e., $u_j = 0$ in $\mathbb{R}^3 \setminus \overline{\Omega}_j$. Let $B_R(0)$ be as before. Applying the divergence theorem to $u_j$ over $B_R(0) \setminus \Omega_j$, we obtain
\begin{equation*}
0 = -\int_{B_R(0) \setminus \Omega_j} u_j \, \mathrm{d}x = \ell_j(\gamma E^s) - \frac{1}{R} \int_{\partial B_R(0)} E^s(x) \cdot x \, \mathrm{d}s(x).
\end{equation*}
Since the surface integral over $\p B_R(0)$ above vanishes, it follows that $\ell_j(\gamma E^s) = 0.$

Combining this with the result for $\ell_j(\gamma E^i)$, we conclude that 
$\ell_j(\varphi) = \ell_j(\gamma E^s) + \ell_j(\gamma E^i) = 0$ for all $j\in\{1,\ldots,J\}.$ This shows that $\varphi \in \mathcal{C}^{1,0}_{\nu,t}(\Gamma)$ also satisfies the modified equation~\eqref{eq:mod_DECFOIE}. Finally, the fact that the unique solution $\widetilde{\varphi}\in C^{0,\alpha}(\Gamma,\C^3)$ of~\eqref{eq:refor} satisfies~\eqref{eq:mod_reg_DECFOIE} follows immediately from the identity $\varphi = \oR_\nu \widetilde{\varphi}$, since
$$
\ell_j(\varphi) = \ell_j(\oR_\nu \widetilde{\varphi}) = 0, \quad j \in \{1, \ldots, J\}.
$$
This completes the proof.
\end{proof}
It remains to determine under which conditions on the stabilization parameter $\xi$ the modified boundary integral equations,~\eqref{eq:mod_DECFOIE} and~\eqref{eq:mod_reg_DECFOIE}, remain uniquely solvable. The following theorem addresses this question:

\begin{theorem} Suppose $\xi\in\C\setminus \bigcup_{j=1}^J\{-|\Gamma_j|^{-1}\}$, where $|\Gamma_j|$ denotes the surface area of the connected component $\Gamma_j$ of $\Gamma$. Then, for all $k > 0$ and $\eta\in\C\setminus\{0\}$, the modified BIEs~\eqref{eq:mod_DECFOIE} and~\eqref{eq:mod_reg_DECFOIE} each admit a unique solution, $\varphi \in \mathcal{C}^{1,0}_{\nu,t}(\Gamma)$ and $\widetilde{\varphi} \in C^{0,\alpha}(\Gamma,\C^3)$, respectively. Moreover, these solutions coincide with the unique solutions of the D-ECFOIE~\eqref{eq:electric_BIE} and the RD-ECFOIE~\eqref{eq:refor}, respectively.\end{theorem}

\begin{proof} 

It suffices to prove the statement for~\eqref{eq:mod_reg_DECFOIE}. It follows from Theorem~\ref{thm:well_poss_E} that the operator $\Id + \oA$ admits a bounded inverse on $C^{0,\alpha}(\Gamma, \mathbb{C}^3)$. Let us define the operator
$$
\oQ : C^{0,\alpha}(\Gamma, \mathbb{C}^3) \to \mathbb{C}^{J}, \qquad \oQ\varphi := [\ell_1\varphi, \ldots, \ell_J\varphi]^\top \in \mathbb{C}^J,
$$
in terms of the linear functionals $\ell_j$ introduced in~\eqref{eq:charge_funcs}, and the operator
$$
\oV : \mathbb{C}^{J} \to C^{0,\alpha}(\Gamma, \mathbb{C}^3), \qquad \oV\mathbf{c} := \sum_{j=1}^J c_j \phi_j, \quad \mathbf{c} = [c_1, \ldots, c_J]^\top \in \mathbb{C}^J,
$$
in terms of the vector fields $\phi_j \in C^{0,\alpha}(\Gamma, \mathbb{C}^3)$ defined in~\eqref{eq:normal_vec_fields}. It is clear that both $\oQ$ and $\oV$ are bounded linear operators on their respective domains.

With these definitions, the modified regularized equation~\eqref{eq:mod_reg_DECFOIE} can be written as
$$
\left( \Id + \oA + \xi \oV \oQ \oR_\nu \right) \widetilde{\varphi} = f.
$$

Applying the Woodbury identity, one finds that the inverse of this operator is given by
\begin{equation}\label{eq:inv_oper}
\left( \Id + \oA + \xi \oV \oQ \oR_\nu \right)^{-1} = (\Id+\oA)^{-1} - \xi (\Id + \oA)^{-1} \oV \left( \Id_J + \xi \Xi \right)^{-1} \oQ \oR_\nu (\Id+\oA)^{-1},
\end{equation}
where $\Id_J$ denotes the identity matrix in $\mathbb{C}^J$ and
$$
\Xi := \oQ \oR_\nu (\Id + \oA)^{-1} \oV=\oQ \oL_e^{-1} \oV \in \mathbb{C}^{J \times J}.
$$

Clearly, the inverse operator~\eqref{eq:inv_oper} exists and is bounded on $C^{0,\alpha}(\Gamma, \mathbb{C}^3)$ if and only if $\Id_J + \xi \Xi$ is invertible. Using the definition of $\phi_j$ in~\eqref{eq:normal_vec_fields}, a direct computation yields
$$
\Xi = \operatorname{diag}\{|\Gamma_1|,\ldots,|\Gamma_J|\},
$$
so the matrix $\Id_J + \xi\Xi$ is invertible if and only if $\xi\in\C\setminus \bigcup_{j=1}^J\{-|\Gamma_j|^{-1}\}$, which completes the proof.
\end{proof}

\begin{remark}
To explain the rationale behind the modified BIEs~\eqref{eq:mod_DECFOIE} and~\eqref{eq:mod_reg_DECFOIE}, note that the modified D-ECFOIE~\eqref{eq:mod_DECFOIE} can be rewritten, using $\phi_j = \oL_e\nu_j$, as
$$
\oL_e\!\left(\Id + \xi\sum_{j=1}^J\nu_j\ell_j\right)\varphi = f,
$$
or equivalently $\left(\Id + \xi\sum_{j=1}^J\nu_j\ell_j\right)\varphi = \oL_e^{-1}f$. Applying the functional $\ell_i$ to both sides and using the fact that $\ell_i(\oL_e^{-1}f)=0$ (which follows from the preceding lemma applied to the datum $f$), we obtain
$$
\ell_i\varphi + \xi\sum_{j=1}^{J}(\ell_i\nu_j)(\ell_j\varphi) = \left(1+\xi|\Gamma_i|\right)\ell_i\varphi = 0,
$$
where we used $\ell_i\nu_j = \delta_{ij}|\Gamma_i|$. Therefore, since $1+\xi|\Gamma_i|\neq 0$ for $\xi\in\C\setminus\bigcup_{j=1}^J\{-|\Gamma_j|^{-1}\}$, the charge conditions $\ell_i\varphi = 0$ are automatically enforced for all $i\in\{1,\ldots,J\}$. In practice, however, the condition $\ell_i(\oL_e^{-1}f)=0$ is only approximately satisfied due to discretization errors and potential matrix ill-conditioning, so the right-hand side above is a small but nonzero residual $\varepsilon_i \approx 0$. In this case the identity becomes $(1+\xi|\Gamma_i|)\ell_i\varphi = \varepsilon_i$, so that $\ell_i\varphi = \varepsilon_i/(1+\xi|\Gamma_i|)$. Choosing $\xi\gg 1$ drives $\ell_i\varphi$ toward zero, thereby penalizing deviations from the charge condition even in the presence of discretization errors.
\end{remark}

\section{Direct magnetic combined-field-only formulation}\label{sec:magnetic_field_BIE}
In this section, we derive a direct boundary integral equation for the magnetic PEC scattering problem~\eqref{eq:mgf_equiv}. Analogously to the electric case, we obtain the following integral representation for the scattered field from Green's representation formula:
\begin{equation*}\label{eq:Green_H}
H^s(\nex) = \mathcal D[\gamma H](x) -\mathcal S[\p_\nu H](\nex), \quad \nex \in \ED,
\end{equation*}
where the traces of the total magnetic field $H = H^s + H^i$, defined in $(\ED) \cap U$, are given by
\begin{equation}\label{eq:total_traces}
\gamma H= \gamma H^s+\gamma H^i\in C^{1,\alpha}(\Gamma,\C^3)\quad\text{and}\quad \p_\nu H = \p_\nu H^s+\p_\nu H^i\in C^{0,\alpha}(\Gamma,\C^3).
\end{equation}

Applying the magnetic PEC boundary conditions~\eqref{eq:H_bc_n} and~\eqref{eq:H_bc_t}, we arrive at
$$
\gamma H = \oP_t\gamma H\quad\text{and}\quad \p_\nu H = \oP_\nu\p_\nu H -\oR\oP_t\gamma H. 
$$

Substituting these identities into Green's representation formula for the magnetic field~\eqref{eq:Green_H} yields
\begin{equation}\label{eq:rep_for_H}
H^s(\nex) = \mathcal D[\oP_t\gamma H](x) -\mathcal S[\oP_\nu\p_\nu H-\oR\oP_t\gamma H](\nex), \quad \nex \in \ED.
\end{equation}

Next, taking exterior traces on both sides of~\eqref{eq:rep_for_H} and using~\eqref{eq:total_traces}, we get
\begin{subequations}\begin{align}
-\gamma H^i =&-\frac12\gamma H+ (\oK+\oS\oR)\oP_t\gamma H - \oS\oP_\nu\p_\nu H,\label{eq:dir_trace_H}\\
-\p_\nu H^i =&-\frac12(\oP_\nu\p_\nu H-\oR\oP_t\gamma H)+ (\oT+\oK'\oR)\oP_t\gamma H - \oK'\oP_\nu\p_\nu H.\label{eq:nue_trace_H}
\end{align}\end{subequations}

Combining these equations as~\eqref{eq:nue_trace_H}$+\im\eta$\eqref{eq:dir_trace_H}, with $\eta \in \mathbb{R} \setminus \{0\}$, we obtain
\begin{equation}\label{eq:proto_MBIE}\begin{aligned}
-\frac12(\oP_\nu\p_\nu H+(\im\eta\Id-\oR)\oP_t\gamma H)+ (\oT+\im\eta\oK+(\oK'+\im\eta \oS)\oR)\oP_t\gamma H - (\oK'+\im\eta\oS)\oP_\nu\p_\nu H=\hspace{1cm}\\
-\p_\nu H^i-\im\eta\gamma H^i
\end{aligned}
\end{equation}

From this, we can identify a suitable BIE, specifically, the direct magnetic combined-field-only integral equation (D-MCFOIE), which is given by:
\begin{subequations}\label{eq:magnetic_BIE}\begin{equation}\begin{aligned}
-\frac12\{\psi_\nu+ (\im\eta\Id- \oR)\psi_t\}+ \{\oT+\im\eta \oK+(\oK'+\im\eta\oS)\oR\}\psi_t - (\oK'+\im\eta\oS)\psi_\nu=g,
 \end{aligned}
 \end{equation}
where the unknown density is
\begin{equation}\label{eq:unknown_density_H}
\psi:=\psi_t+\psi_\nu\in\mathcal C^{0,1}_{\nu,t},\quad\psi_t:=\oP_t(\gamma H),\quad  \psi_\nu:=\oP_\nu(\p_\nu H),
\end{equation}
and the right-hand-side is given by 
\begin{equation}\label{eq:g_datum}
 g:=-\p_\nu H^i-\im\eta\gamma H^i\in C^{0,\alpha}(\Gamma,\C^3).
\end{equation}\end{subequations}

The scattered magnetic field can then be retrieved from the BIE solution via the representation formula~\eqref{eq:rep_for_H}.

Leveraging the surface projection operators, the D-MCFOIE can be recast in the form
\begin{equation}\label{eq:L_m_op}
\oL_m\psi = g
\end{equation}
where $\oL_m: \mathcal C^{0,1}_{\nu,t}(\Gamma)\to C^{0,\alpha}(\Gamma,\C^3)$
is the operator defined by
\begin{equation}\label{eq:L_magnetic}
\oL_m := -\tfrac12\{\oP_\nu+ (\im\eta\Id- \oR)\oP_t\}+ \{\oT+\im\eta \oK+(\oK'+\im\eta\oS)\oR\}\oP_t - (\oK'+\im\eta\oS)\oP_\nu.
\end{equation}

We now proceed to prove the uniqueness of solutions to~\eqref{eq:magnetic_BIE}.

\begin{theorem}\label{eq:unique_H} The D-MCFOIE given by~\eqref{eq:magnetic_BIE} admits at most one solution $\psi \in \mathcal{C}_{\nu,t}^{0,1}(\Gamma)$ for all wavenumbers $k > 0$ and all coupling parameters $\eta \in \mathbb{R} \setminus \{0\}$.
\end{theorem}

\begin{proof} We proceed as in the proof of Theorem~\ref{eq:unique_E}. 
Suppose there exists a non-trivial vector density $\psi \in \mathcal{C}^{0,1}_{\nu,t}(\Gamma)$ satisfying the homogeneous BIE:
\begin{equation}\label{eq:homo_H}\begin{aligned}
-\frac12\{\psi_\nu+ (\im\eta\Id- \oR)\psi_t\}+ \{\oT+\im\eta \oK+(\oK'+\im\eta\oS)\oR\}\psi_t - (\oK'+\im\eta\oS)\psi_\nu=0.
 \end{aligned}
\end{equation}
Letting  $\psi = \psi_\nu + \psi_t$, where $\psi_\nu = \oP_\nu \psi$ and $\psi_t = \oP_t \psi$, we define the corresponding layer potential as follows:
$$
V(x) :=  \mathcal D[\psi_t](x) - \mathcal S[\psi_\nu - \oR\psi_t](x),\quad x\in\R^3\setminus\Gamma.
$$

Taking traces of the potential, we obtain
\begin{equation}\label{eq:jump_tr_H}
\gamma^\pm V = \pm\frac{1}{2} \psi_t + (\oK + \oS \oR)\psi_t - \oS \psi_\nu, \quad \text{and} \quad \p_\nu^\pm V = \pm\frac{1}{2}(\psi_\nu - \oR \psi_t) + (\oT + \oK' \oR)\psi_t - \oK' \psi_\nu,
\end{equation}
from which it follows, together with~\eqref{eq:homo_H}, that $\p_\nu^- V + \im \eta \gamma^- V = 0$. Since the restriction $V|_{\Omega} \in C^{2}(\Omega,\C^3) \cap C^{1,\alpha}(\overline\Omega,\C^3)$ of the layer potential $V$ to $\Omega$ satisfies the Helmholtz equation $\Delta V|_{\Omega} + k^2 V|_{\Omega} = 0$ in $\Omega$ and the Robin boundary condition with an imaginary coupling parameter, the uniqueness of this interior problem implies that $\gamma^- V = 0$ and $\p_\nu^- V = 0$.

Substituting these identities into~\eqref{eq:jump_tr_H}, we find
$$
\gamma^+ V = \psi_t, \quad \text{and} \quad \p_\nu^+ V = \psi_\nu - \oR \psi_t.
$$
Now, using the facts that $\oP_\nu \psi_t = 0$, $\oP_t \psi_\nu = 0$, and $\oP_t \oR = \oR \oP_t$, we get
\begin{equation}\label{eq:homo_mag}
\oP_\nu \gamma^+ V = 0, \quad \text{and} \quad \oP_t (\p_\nu^+ V + \oR \gamma^+ V) = 0.
\end{equation}

Since in addition to~\eqref{eq:homo_mag} the restriction $V|_{\ED}\in C^{2}(\ED,\C^3)\cap C^{1,\alpha}(\R^3\setminus\Omega,\C^3)$ satisfies both the Helmholtz equation in the exterior domain~\eqref{eq:vec_Helm_eqn_H} and the Sommerfeld radiation condition~\eqref{eq:sommerfeld_condition_H}, it follows from Theorem~\ref{eq:equiv} that $V|_{\ED}$ solves the exterior PEC scattering problem~\eqref{eq:scatt_problem} with homogeneous boundary conditions. As this problem admits a unique solution by Theorem~\ref{eq:reg_solution}, we conclude that $V = 0$ in $\mathbb{R}^3 \setminus \Omega$. This leads to a contradiction, since $\psi = \oP_t\gamma^+ V + \oP_\nu\p^+_\nu V = 0$. The proof is thus complete.

\end{proof}

As in the case of the D-ECFOIE, we rely on a Calderón-type regularization of the magnetic boundary integral equation~\eqref{eq:magnetic_BIE} to establish its Fredholm properties and, subsequently, the existence of solutions. In this case, however, proving the invertibility of the corresponding regularizing operator is slightly more involved, so we state this result in the following lemma.

\begin{lemma}\label{lemm:precond_H} The operator
\begin{equation}\label{eq:precond_t}
\oR_t:\mathcal{C}^{0,1}_{\nu,t}(\Gamma) \to C^{0,\alpha}(\Gamma,\mathbb{C}^3), \quad \oR_t\psi = -2(\oP_\nu + 2\oP_t \oS_0 \oP_t)\psi,
\end{equation}
is a bounded linear operator that admits a bounded inverse.
\end{lemma}
\begin{proof}  Clearly, $\oR_t$ is linear and bounded. We then proceed to show that $\oR_t$ is injective. Suppose there exists a nontrivial function $\psi \in C^{0,\alpha}(\Gamma, \mathbb{C}^3)$ such that
\begin{equation}\label{eq:homo_reg}
-\frac12\oR_t\psi = \psi_\nu + 2\oP_t \oS_0 \psi_t = 0,
\end{equation}
where we let $\psi_\nu := \oP_\nu \psi$ and $\psi_t := \oP_t \psi$.

Let us denote by
$$
(\varphi, \psi)_{L^2(\Gamma, \mathbb{C}^3)} := \int_\Gamma \varphi \cdot \overline{\psi} \de s
$$
the inner product on the Hilbert space $L^2(\Gamma, \mathbb{C}^3)$, equipped with the norm
$\|\psi\|_{L^2(\Gamma, \mathbb{C}^3)} := \sqrt{(\psi, \psi)_{L^2(\Gamma, \mathbb{C}^3)}}.$

Taking the inner product with $\psi$ on both sides of~\eqref{eq:homo_reg}, we obtain
\begin{align*}
(\psi, \oP_\nu \psi)_{L^2(\Gamma,\C^3)} + 2(\psi, \oP_t \oS_0 \oP_t \psi)_{L^2(\Gamma,\C^3)}= (\psi_\nu, \psi_\nu)_{L^2(\Gamma,\C^3)} + 2(\psi_t, \oS_0 \psi_t)_{L^2(\Gamma,\C^3)} \
&= 0.
\end{align*}
Extracting the real part of the identity above yields
\begin{equation}\label{eq:vanish_R_t}
\|\psi_\nu\|_{L^2(\Gamma,\C^3)}^2 + 2\real(\psi_t, \oS_0 \psi_t)_{L^2(\Gamma,\C^3)} = 0.
\end{equation}

Then, in view of the facts that $\real(\psi_t, \oS_0 \psi_t)_{L^2(\Gamma,\C^3)} \geq 0$, and that equality holds if and only if $\psi_t=0$, since $\psi_t \in C^{0,\alpha}(\Gamma,\C^3) \subset L^2(\Gamma,\C^3) \subset H^{-1/2}(\Gamma,\C^3)$~\cite[Thm. 8.10]{Mclean2000Strongly}, we conclude that both terms on the left-hand side of~\eqref{eq:vanish_R_t} must vanish.
Therefore,
$\psi_\nu = 0$ and   $\psi_t = 0$, which implies that $\psi = 0$, contradicting the assumption that $\psi$ is nontrivial. We then conclude that $\oR_t$ is injective. 

To finish the proof, it suffices to show that for all $\psi\in \mathcal C^{0,1}_{\nu,t}(\Gamma)$ there exists a unique $\widetilde\psi\in C^{0,\alpha}(\Gamma,\C^3)$ such that
\begin{equation}\label{eq:precond_eq}
-\frac{1}{2}\oR_t\widetilde\psi = (\oP_\nu + 2\oP_t\oS_0\oP_t)\widetilde\psi = \psi.
\end{equation}

It follows directly from~\eqref{eq:precond_eq} and the mapping properties of $\oP_\nu$ that the normal component $\widetilde\psi_\nu := \oP_\nu\widetilde\psi = \oP_\nu\psi \in C^{0,\alpha}(\Gamma,\C^3)$. To prove the existence of the tangential component $\widetilde\psi_t := \oP_t\widetilde\psi$, we rewrite~\eqref{eq:precond_eq} as
\begin{equation*}
\oP_\nu\psi + 2(\oS_0 + \oC_t)\widetilde\psi_t = \psi,
\end{equation*}
in terms of the commutator $\oC_t$ defined in~\eqref{eq:C_t}. Subtracting $\oP_\nu\psi$ from both sides and using that $\oP_t\psi = \psi - \oP_\nu\psi \in C^{1,\alpha}(\Gamma,\C^3)$, we obtain
$$
(\oS_0 + \oC_t)\widetilde\psi_t = \tfrac{1}{2}\oP_t\psi.
$$

Applying $\oS_0^{-1} : C^{1,\alpha}(\Gamma, \C^3) \to C^{0,\alpha}(\Gamma, \C^3)$ and using that this operator is bounded~\cite[Thm. 7.40]{kress2012linear}, we get
\begin{equation}\label{eq:ref_precond_2}
\widetilde\psi_t + \oS_0^{-1}\oC_t \widetilde\psi_t = \tfrac{1}{2} \oS_0^{-1} \oP_t \psi.
\end{equation}

Since $\oC_t : C^{0,\alpha}(\Gamma,\C^3) \to C^{1,\alpha}(\Gamma,\C^3)$ is compact by Lemma~\ref{lem:comm_bounded}, the composition $\oS_0^{-1}\oC_t$ is compact on $C^{0,\alpha}(\Gamma,\C^3)$. Hence, in view of the injectivity of $\oR_t$ and the fact that $\frac{1}{2} \oS_0^{-1} \oP_t \psi \in C^{0,\alpha}(\Gamma,\C^3)$, the Fredholm alternative guarantees the existence and uniqueness of a solution $\widetilde\psi_t \in C^{0,\alpha}(\Gamma,\C^3)$ to~\eqref{eq:ref_precond_2} which corresponds to the tangential component of $\widetilde\psi$. This completes the proof.

\end{proof}

We are now in a position to establish the well-posedness of the magnetic BIE~\eqref{eq:magnetic_BIE}.

\begin{theorem}\label{thm:well_poss_H} The D-MCFOIE~\eqref{eq:magnetic_BIE} is uniquely solvable in $\mathcal{C}^{0,1}_{\nu,t}(\Gamma)$ for all wavenumbers $k>0$ and coupling parameters $\eta\in\R\setminus\{0\}$.
\end{theorem}
\begin{proof} 
Consider the regularized BIE resulting from substituting $\psi = \oR_t \widetilde\psi$ into~\eqref{eq:magnetic_BIE}, where $\oR_t$ is the regularizing operator defined in~\eqref{eq:precond_t}. This substitution leads to the equation
\begin{equation}\label{eq:regularized_magnetic_BIE}
\widetilde\psi_\nu+2(\im\eta\Id- \oR)\oP_t\oS_0\widetilde\psi_t - 4\{\oT+\im\eta \oK+(\oK'+\im\eta\oS)\oR\}\oP_t\oS_0\widetilde\psi_t +2 (\oK'+\im\eta\oS)\widetilde\psi_\nu=g,
\end{equation}
where we seek a solution $\widetilde\psi=\widetilde\psi_\nu+\widetilde\psi_t \in C^{0,\alpha}(\Gamma, \C^3)$ with $\widetilde\psi_\nu:=\oP_\nu\widetilde\psi$ and $\widetilde\psi_t:=\oP_t\widetilde\psi$. 

It then follows from the invertibility of $\oR_t$ (Lemma~\ref{lemm:precond_H}) and the uniqueness of solutions to~\eqref{eq:magnetic_BIE} (Theorem~\ref{eq:unique_H}) that, to complete the proof, it suffices to show that the regularized equation~\eqref{eq:regularized_magnetic_BIE} is of the form identity plus a compact perturbation in $C^{0,\alpha}(\Gamma, \C^3)$. This is so because, by the Fredholm alternative, uniqueness implies existence when the operator is of Fredholm type with index zero.


To this end, we substitute the identities~\eqref{eq:follow_calderon} and $\oP_t\oS_0\oP_t\widetilde\psi = \oS_0\oP_t\widetilde\psi + \oC_t\oP_t\widetilde\psi$ into~\eqref{eq:regularized_magnetic_BIE}, which allows us to rewrite the equation in the form 
\begin{equation}\label{eq:reg_mag_BIE}
\widetilde\psi + \oB\widetilde\psi = g,
\end{equation}
where $\oB:=\sum_{j=1}^4\oB_j$ and the operators $\oB_j:C^{0,\alpha}(\Gamma,\C^3)\to C^{0,\alpha}(\Gamma,\C^3)$ for $j\in\{1,2,3,4\}$ are given by 
\begin{align*}
\oB_1\psi: =&~ 2(\im\eta\Id- \oR)\oP_t\oS_0\psi,\\
\oB_2\psi:=&-4\{{\oK'_0}^2+(\oT-\oT_0)\oS_0+\im\eta\oK\oS_0+(\oK'+\im\eta\oS)\oR\oS_0\}\oP_t\psi,\\
\oB_3\psi:=&-4\{\oT+\im\eta\oK+(\oK'+\im\eta\oS)\oR\}\oC_t\oP_t\psi ,\\
\oB_4\psi:=&~2 (\oK'+\im\eta\oS)\oP_\nu\psi.
\end{align*}

Following the arguments in the proof of Theorem~\ref{thm:well_poss_E}, it is straightforward to show that each of the operators $\oB_j$ is compact on $C^{0,\alpha}(\Gamma,\C^3)$. In particular, $\oB_3$ is compact because it involves the composition $\oT \oC_t$, where $\oC_t : C^{0,\alpha}(\Gamma,\C^3) \to C^{1,\alpha}(\Gamma,\C^3)$ is compact, as established in Lemma~\ref{lem:comm_bounded}, and $\oT : C^{1,\alpha}(\Gamma,\C^3) \to C^{0,\alpha}(\Gamma,\C^3)$ is bounded. It follows that $\oB$ is compact on $C^{0,\alpha}(\Gamma,\C^3)$, which completes the proof.\end{proof}

\begin{remark}
Theorem~\ref{thm:well_poss_H} can be extended to the zero-frequency limit $k=0$ for a simply connected surface $\Gamma$. In that case, the magnetostatic exterior problem---the $k=0$ limit of~\eqref{eq:mgf_equiv}---is uniquely solvable without additional constraints~\cite{werner1963perfect_reflection}, in contrast to the electrostatic case which requires prescribing the charge integrals. This unique solvability can be leveraged, following the same arguments as in~\cite{burbano2025maxwell} for the indirect (R-)MCFOIE, to show that the D-MCFOIE remains well-posed at $k=0$ for any $\eta\in\R\setminus\{0\}$. We do not pursue this extension here, as the proof is analogous to that in~\cite{burbano2025maxwell}.
\end{remark}

We conclude this section by noting that the D-MCFOIE~\eqref{eq:magnetic_BIE} solution $\psi\in\mathcal C^{0,1}_{\nu,t}(\Gamma)$ provides a direct representation of the surface electric currents:
$$
J := \nu \times \gamma H = \nu \times \oP_t \psi.
$$
These surface currents are of practical importance in their own right, playing a central role in various applications. Moreover, they can be directly used to reconstruct both the magnetic and electric components of the scattered electromagnetic field via the Stratton--Chu integral representation formulae~\cite[Thm. 3.30]{kirsch2016mathematical}:
\begin{align}\label{eq:stratton-chu}
E^s(x) = -\frac{1}{\im\omega \epsilon} \curl\curl \int_{\Gamma} J(y) G(x, y) \de s(y)\quad\text{and}\quad
H^s(x) = \curl \int_{\Gamma} J(y) G(x, y) \de s(y).
\end{align}

\section{Numerical examples\label{sec:numerics}}
This section presents numerical experiments validating all four BIE formulations introduced in this paper: the direct electric and magnetic formulations D-ECFOIE~\eqref{eq:electric_BIE} and D-MCFOIE~\eqref{eq:magnetic_BIE}, and their analytically preconditioned counterparts RD-ECFOIE~\eqref{eq:refor} and RD-MCFOIE~\eqref{eq:reg_mag_BIE}, which employ the right preconditioners $\oR_\nu$~\eqref{eq:precond_nu} and $\oR_t$~\eqref{eq:precond_t}, respectively. More specifically, we solve the equations $\mathsf L_e \varphi = f$ and $\mathsf L_m\psi = g$ and the preconditioned equations $\mathsf L_e\mathsf R_\nu\varphi =f$ and $\mathsf L_m\mathsf R_t\psi = g$, where $\mathsf L_e$ and $\mathsf L_m$ are the operators defined in~\eqref{eq:L_electric} and~\eqref{eq:L_magnetic}, respectively, and the data $f$ and $g$ are given by~\eqref{eq:f_datum} and~\eqref{eq:g_datum}, respectively. The scattered fields are then evaluated via the representation formulae~\eqref{eq:rep_for_E} and~\eqref{eq:rep_for_H} for the electric and magnetic cases, respectively.

All experiments use a Nyström discretization based on the General-Purpose Density Interpolation Method (GP-DIM)~\cite{faria2021general}, which builds on the earlier DIM framework~\cite{HDI3D,perez2019planewave,perez2020planewave}, as implemented in the open-source Julia package \texttt{Inti.jl}~\cite{IntegralEquations_Inti_2025}. All linear systems are solved with GMRES~\cite{saad1986gmres}, with $\mathcal{H}$-matrix compression through \texttt{HMatrices.jl} for efficient operator assembly and application. Surface meshes are generated with Gmsh~\cite{geuzaine2009gmsh} using curved triangular elements of degree five, which provide sufficient geometric accuracy for the curvature quantities $\mathscr{H}$ and $\mathscr{R}$ entering the boundary conditions; while the resulting meshes are not globally $C^{2,\alpha}$-smooth in the strict sense required by the analysis, they are smooth enough to ensure convergence in practice.

The incident field throughout this section is a planewave
\begin{subequations}\begin{align}
        E^i(x)&:=p\e^{\im kx\cdot d},& \p_\nu E^i(x) &= \im k(d\cdot\nu(x))\gamma E^i(x),\\
    H^i(x)&:=\sqrt{\tfrac{\epsilon}{\mu}}(d\times p)\e^{\im kx\cdot d},& \p_\nu H^i(x) &= \im k(d\cdot\nu(x))\gamma H^i(x),\quad x\in\Gamma,
\end{align}\label{eq:planewave}\end{subequations}
where $p, d \in \mathbb{R}^3\setminus\{0\}$ are the polarization and propagation directions, with $|d| = 1$ and $p \cdot d = 0$, and we set $\epsilon = \mu = 1$.

Numerical errors are assessed using the relative error measures defined by
\begin{equation}\label{eq:error_measures}\begin{split}
e_{F} := \max_{j\in\{1,\ldots,100\}} \frac{|\widetilde F(x_j) - F_{\rm ref}(x_j)|}{|F_{\rm ref}(x_j)|}\quad\text{and}\quad
e_{\dive\!F} := \max_{j\in\{1,\ldots,100\}} \frac{|\dive \widetilde F(x_j)|}{|\widetilde F(x_j)|},
\end{split}\end{equation}
where $F_{\rm ref}$ denotes the exact (electric or magnetic) reference solution, and $\widetilde F$ is the approximate field computed via quadrature evaluation of the corresponding representation formula,~\eqref{eq:rep_for_E} for electric field and~\eqref{eq:rep_for_H} for magnetic field. The target points $x_j$, for $j = 1,\ldots,100$, are (approximately) uniformly distributed on a sphere of radius 5 centered at the origin, ensuring it encloses the surface $\Gamma$ in all test cases. 

\begin{table}[htb]
\centering
\footnotesize
\renewcommand{\arraystretch}{1.0}
\begin{tabular}{ccccccccc}
  \toprule
\multicolumn{9}{c}{Sphere}\\
\midrule
\multicolumn{2}{c}{} & \multicolumn{4}{c}{D-ECFOIE~\eqref{eq:electric_BIE}} & \multicolumn{3}{c}{D-MCFOIE~\eqref{eq:magnetic_BIE}} \\
\cmidrule(lr){3-6}\cmidrule(lr){7-9}
$\lambda/h$ & $N$ & $e_{E^s}$ & $e_{\dive E^s}$ & $|q_1|$ & iter. & $e_{H^s}$ & $e_{\dive H^s}$ & iter.\\ \midrule
$2^{2.0}$  &   1540 & 7.80e-3 & 2.26e-2 & 6.85e-4 & 152 & 1.24e-2 & 1.56e-2 & 125 \\
$2^{2.5}$  &   2560 & 2.80e-3 & 8.56e-3 & 5.83e-5 & 162 & 2.65e-3 & 2.95e-3 & 127 \\
$2^{3.0}$  &   5400 & 2.49e-4 & 5.05e-4 & 1.79e-5 & 121 & 2.22e-4 & 2.20e-4 &  96 \\
$2^{3.5}$  &  10860 & 7.96e-5 & 1.68e-4 & 2.25e-6 &  51 & 3.68e-5 & 4.68e-5 &  39 \\
$2^{4.0}$  &  21160 & 1.52e-5 & 3.33e-5 & 5.74e-7 &  38 & 8.37e-6 & 6.97e-6 &  31 \\
$2^{4.5}$  &  39920 & 4.05e-6 & 8.96e-6 & 1.05e-7 &  40 & 2.60e-6 & 1.83e-6 &  34 \\
$2^{5.0}$  &  79720 & 1.88e-6 & 4.87e-6 & 2.73e-8 &  41 & 7.89e-7 & 1.08e-6 &  35 \\
$2^{5.5}$  & 157860 & 9.10e-7 & 2.79e-6 & 2.21e-8 &  50 & 3.22e-7 & 5.16e-7 &  44 \\
$2^{6.0}$  & 310480 & 7.86e-7 & 2.40e-6 & 1.17e-8 &  54 & 4.69e-7 & 5.53e-7 &  44 \\
\midrule\addlinespace[3pt]
\multicolumn{2}{c}{} & \multicolumn{4}{c}{RD-ECFOIE~\eqref{eq:refor}} & \multicolumn{3}{c}{RD-MCFOIE~\eqref{eq:reg_mag_BIE}} \\
\cmidrule(lr){3-6}\cmidrule(lr){7-9}
$\lambda/h$ & $N$ & $e_{E^s}$ & $e_{\dive E^s}$ & $|q_1|$ & iter. & $e_{H^s}$ & $e_{\dive H^s}$ & iter.\\ \midrule
$2^{2.0}$ &   1540 & 7.79e-3 & 2.25e-3 & 6.85e-4 & 215 & 1.24e-2 & 1.56e-2 & 199 \\
$2^{2.5}$ &   2560 & 2.80e-3 & 8.56e-3 & 5.83e-5 & 304 & 2.65e-3 & 2.94e-3 & 309 \\
$2^{3.0}$ &   5400 & 2.50e-4 & 5.07e-4 & 1.79e-5 & 260 & 2.25e-4 & 2.26e-4 & 216 \\
$2^{3.5}$ &  10860 & 7.90e-5 & 1.68e-4 & 2.22e-6 &  95 & 3.73e-5 & 4.61e-5 &  34 \\
$2^{4.0}$ &  21160 & 1.72e-5 & 4.07e-5 & 5.76e-7 &  30 & 8.07e-6 & 7.40e-6 &  24 \\
$2^{4.5}$ &  39920 & 3.94e-6 & 7.89e-6 & 9.47e-8 &  25 & 2.67e-6 & 2.67e-6 &  21 \\
$2^{5.0}$ &  79720 & 1.40e-6 & 2.64e-6 & 2.39e-8 &  25 & 1.11e-6 & 1.91e-6 &  21 \\
$2^{5.5}$ & 157860 & 7.04e-7 & 1.60e-6 & 3.15e-8 &  25 & 9.55e-7 & 1.71e-6 &  21 \\
$2^{6.0}$ & 310480 & 1.01e-6 & 2.66e-6 & 8.25e-9 &  25 & 8.99e-7 & 2.03e-6 &  21 \\
\bottomrule
\end{tabular}
\caption{
Accuracy and GMRES iteration counts (with a relative tolerance of $10^{-6}$) for the numerical solution of all four direct BIE formulations for electromagnetic scattering by a PEC sphere $\Gamma=\mathbb S^2$, illuminated by a plane wave~\eqref{eq:planewave} with wavenumber $k = \pi$. The top half shows the direct formulations and the bottom half their analytically preconditioned counterparts. The electric formulations (left) report errors for the scattered electric field $\widetilde{E}^s$; the magnetic formulations (right) report errors for $\widetilde{H}^s$. The results are computed for various target mesh sizes $h > 0$, leading to surface discretizations with $N$ quadrature nodes (10 points per element) and linear systems of size $3N$. The parameter $\eta = 100k$ was used throughout and $\xi=0$ in the case of the (R)D-ECFOIE.
}
\label{tab:planewave_sphere}
\end{table}

Exact solutions are generally unavailable; an exception is the PEC unit sphere $\Gamma = \mathbb{S}^2$ considered in the next section, for which the Mie series solution provides an analytical reference computed via the Julia package \texttt{SphericalScattering.jl}~\cite{hofmann2023sphericalscattering}. In all the other cases we rely on divergence error, which exhibits the same order of convergence as the field error and is a good proxy for it, as a measure of accuracy. (This is so because the divergence-free condition is enforced via a boundary condition in the BIE formulations.) In addition to the error measures above, in the case of the electric-field BIEs we additionally monitor the approximate surface charge integrals
\begin{equation}\label{eq:surf_charge}
q_j := \int_{\Gamma_j} \nu \cdot \gamma E^s\,\de s, \quad j\in\{1,\ldots,J\},
\end{equation}
which should vanish for all $k>0$. Finally, we report the number of GMRES iterations required to reach a relative residual tolerance specified in each case (see the captions of the tables for details).
\begin{remark}
The divergence errors reported in the tables can be suppressed by applying the field correction
$$
\widetilde F \mapsto \widetilde F + \frac{1}{k^2}\nabla \dive \widetilde F,
$$
also employed in the indirect formulations of~\cite{burbano2025maxwell}. Alternatively, when the (R)D-MCFOIE is used, the surface current $J = \nu\times\gamma H$ can be reconstructed from the BIE solution and used to evaluate the scattered fields via the Stratton--Chu formulae~\eqref{eq:stratton-chu}, which yield divergence-free fields by construction. Throughout this section, however, we report errors for the uncorrected fields $\widetilde F$.
\end{remark}

\begin{table}[htb]
\centering
 \footnotesize
 \renewcommand{\arraystretch}{1.0}
\begin{tabular}{cccccccccccc}
  \toprule
\multicolumn{12}{c}{Sphere}\\
\midrule
& \multicolumn{4}{c}{D-ECFOIE~\eqref{eq:electric_BIE}; $\xi = 0$} & \multicolumn{4}{c}{D-ECFOIE~\eqref{eq:mod_DECFOIE}; $\xi=\pi\cdot10^{4}$} & \multicolumn{3}{c}{D-MCFOIE~\eqref{eq:magnetic_BIE}} \\
\cmidrule(lr){2-5}\cmidrule(lr){6-9}\cmidrule(lr){10-12}
$\lambda/d$ & $e_{E^s}$ & $e_{\dive\! E^s}$ & $|q_1|$ & \#iter & $e_{E^s}$ & $e_{\dive\! E^s}$ & $|q_1|$ & \#iter & $e_{H^s}$ & $e_{\dive\! H^s}$ & \#iter \\ \midrule
$10^{16}$ & 1.76e-5 & 2.11e-6 & 2.09e-6 &  4 & 1.76e-5 & 2.11e-6 & 2.06e-6 &  4 & 3.56e-4 & 3.44e-5 &  2 \\
$10^{08}$ & 1.75e-5 & 2.10e-6 & 2.08e-6 &  4 & 1.75e-5 & 2.10e-6 & 2.05e-6 &  4 & 3.56e-4 & 3.44e-5 &  2 \\
$10^{04}$ & 7.38e-6 & 8.74e-7 & 1.00e-6 &  5 & 7.38e-6 & 8.73e-7 & 1.04e-6 &  5 & 3.58e-4 & 3.47e-5 &  4 \\
$10^{02}$ & 7.06e-6 & 8.13e-7 & 3.16e-6 &  8 & 6.61e-6 & 8.17e-7 & 1.08e-9 &  9 & 3.58e-4 & 3.41e-5 &  6 \\
$10^{01}$ & 7.51e-6 & 9.04e-7 & 3.49e-6 & 11 & 7.02e-6 & 6.48e-7 & 4.40e-12 & 12 & 3.05e-4 & 1.85e-5 & 10 \\
$10^{00}$ & 2.35e-6 & 7.06e-6 & 1.09e-7 & 37 & 2.76e-6 & 8.32e-6 & 2.75e-13 & 38 & 7.62e-6 & 6.36e-6 & 31 \\
\midrule\addlinespace[3pt]
 & \multicolumn{4}{c}{RD-ECFOIE~\eqref{eq:refor}; $\xi = 0$} & \multicolumn{4}{c}{RD-ECFOIE~\eqref{eq:mod_reg_DECFOIE}; $\xi=\pi\cdot10^{4}$} & \multicolumn{3}{c}{RD-MCFOIE~\eqref{eq:reg_mag_BIE}} \\
\cmidrule(lr){2-5}\cmidrule(lr){6-9}\cmidrule(lr){10-12}
$\lambda/d$ & $e_{E^s}$ & $e_{\dive\! E^s}$ & $|q_1|$ & \#iter & $e_{E^s}$ & $e_{\dive\! E^s}$ & $|q_1|$ & \#iter & $e_{H^s}$ & $e_{\dive\! H^s}$ & \#iter \\ \midrule
$10^{16}$ & 1.62e-5 & 1.20e-6 & 2.85e-5 &  5 & 1.40e-5 & 1.14e-6  & 2.33e-5 &  5 & 3.58e-4 & 3.46e-5 &  2 \\
$10^{08}$ & 1.81e-5 & 1.74e-6 & 2.58e-5 &  5 & 1.63e-5 & 1.67e-6 & 2.21e-5 &  5 & 3.58e-4 & 3.46e-5 &  2 \\
$10^{04}$ & 1.85e-5 & 1.26e-6 & 3.23e-5 &  7 & 2.02e-5 &1.09e-6 & 4.12e-5 &  7 & 3.58e-4 & 3.48e-5 &  5 \\
$10^{02}$ & 2.70e-5 & 1.21e-6 & 5.65e-5 &  9 & 1.10e-5 & 1.23e-6 & 1.04e-9 & 10 & 3.58e-4 & 3.42e-5 &  6 \\
$10^{01}$ & 2.19e-5 & 4.79e-6 & 2.34e-6 & 13 & 9.89e-6 &  9.75e-7 & 4.17e-12 &12& 3.07e-4 & 1.87e-5 &  9 \\
$10^{00}$ & 2.13e-6 & 6.23e-6 & 1.33e-7 & 28 & 2.57e-6 & 7.42e-6 & 3.32e-13 & 29 & 7.86e-6 & 7.86e-6 & 24 \\
\bottomrule
\end{tabular}
\caption{Relative errors, induced surface charge~\eqref{eq:surf_charge}, and GMRES iteration counts (with a relative tolerance of $10^{-6}$) for all four direct BIE formulations applied to the electromagnetic scattering problem by the unit sphere $\Gamma = \mathbb{S}^2$ (of diameter $d=2$) under planewave illumination~\eqref{eq:planewave} at low frequencies. The results are obtained using a fixed surface mesh of size $h \approx 0.16$. In all cases, the parameter $\eta = 100\pi$ is used, and two different values of the stabilization parameter $\xi$---introduced to mitigate the low-frequency breakdown---are shown.
}
\label{tab:low_frequency}
\end{table}

\begin{table}[htb]
\centering
\footnotesize
\renewcommand{\arraystretch}{1.0}
\begin{tabular}{ccccccccc}
\toprule
\multicolumn{9}{c}{Sphere}\\
\midrule
\multicolumn{2}{c}{} & \multicolumn{4}{c}{D-ECFOIE~\eqref{eq:electric_BIE}} & \multicolumn{3}{c}{D-MCFOIE~\eqref{eq:magnetic_BIE}} \\
\cmidrule(lr){3-6}\cmidrule(lr){7-9}
$k/\pi$ & $N$ & $e_{E^s}$ & $e_{\dive E^s}$ & $|q_1|$ & iter. & $e_{H^s}$ & $e_{\dive H^s}$ & iter.\\ \midrule
1  &  12696  & 7.66e-5  & 2.10e-4  & 2.43e-5  &  56  & 4.47e-5  & 6.75e-5  &  45 \\
2  &   47832 & 3.53e-5  & 1.21e-4  & 1.46e-6  &  97  & 1.57e-5  & 2.62e-5  &  78 \\
3  &  105648 & 4.26e-6  & 3.84e-5  & 5.05e-7  & 157  & 2.86e-6  & 2.18e-5  & 121 \\
4  &  186288 & 2.94e-6  & 3.25e-5  & 9.63e-8  & 274  & 2.96e-6  & 8.52e-6  & 220 \\
5  &  289212 & 5.45e-6  & 7.48e-5  & 2.30e-7  & 327  & 5.59e-6  & 1.68e-5  & 280 \\
\midrule\addlinespace[3pt]
\multicolumn{2}{c}{} & \multicolumn{4}{c}{RD-ECFOIE~\eqref{eq:refor}} & \multicolumn{3}{c}{RD-MCFOIE~\eqref{eq:reg_mag_BIE}} \\
\cmidrule(lr){3-6}\cmidrule(lr){7-9}
$k/\pi$ & $N$ & $e_{E^s}$ & $e_{\dive E^s}$ & $|q_1|$ & iter. & $e_{H^s}$ & $e_{\dive H^s}$ & iter.\\ \midrule
1 &   12696  & 7.46e-5  & 2.04e-4  & 2.45e-5  &  49  & 5.34e-5  & 5.14e-5  &  38 \\
2 &   47832  & 3.62e-5  & 1.35e-4  & 1.46e-6  &  67  & 2.21e-5  & 2.86e-5  &  57 \\
3 &  105648  & 6.85e-6  & 6.20e-5  & 4.97e-7  &  99  & 1.14e-5  & 2.52e-6  & 107 \\
4 &  186288  & 3.34e-6  & 3.95e-5  & 9.77e-8  & 129  & 1.27e-5  & 5.84e-6  & 119 \\
5 &  289212  & 7.85e-6  & 9.18e-5  & 2.22e-7  & 142  & 7.95e-6  & 4.05e-5  & 136 \\
\bottomrule
\end{tabular}
\caption{
Frequency robustness of all four BIE formulations for planewave scattering~\eqref{eq:planewave} by the PEC sphere $\Gamma=\mathbb S^2$ at wavenumbers $k\in\{\pi,2\pi,3\pi,4\pi,5\pi\}$. The mesh is adapted to each frequency to maintain $h\approx\lambda/10$, so that the number of quadrature nodes $N$ scales with $k$; linear systems have size $3N$, with six quadrature points per element. Layout, error measures, GMRES tolerance ($10^{-6}$), and parameters ($\eta=100k$, $\xi=0$) are as in Table~\ref{tab:planewave_sphere}.}
\label{tab:higher_frequencies_sphere}
\end{table}


Our first set of examples considers scattering of a plane wave~\eqref{eq:planewave} ($p=(1,0,0)$, $d=(0,0,1)$) by the unit PEC sphere at wavenumber $k=\pi$ ($\lambda=2$). Table~\ref{tab:planewave_sphere} reports errors~\eqref{eq:error_measures} and GMRES iteration counts for all four formulations against the Mie series reference, over a range of nearly uniform surface discretizations with meshsize $h>0$.

All formulations achieve the expected accuracy using GP-DIM, confirming the viability of BIEs formulated entirely in terms of Helmholtz operators. The regularized formulations (RD-ECFOIE and RD-MCFOIE) require fewer GMRES iterations than their unregularized counterparts, with counts remaining essentially flat under mesh refinement while a mild growth is observed for the direct formulations---showcasing the effect of the Calderón-type preconditioning. The improvement is modest, partly because the large coupling parameter $\eta=100k$ already reduces ill-conditioning at the unregularized level by amplifying the compact operators $\oS$ and $\oK$ relative to the hypersingular $\oT$. This choice also improves accuracy, since the GP-DIM evaluation of $\oT$ and $\oK'$ is less accurate than that of $\oS$ and $\oK$. (Smaller values of $\eta$ result in significantly larger GMRES iteration counts across all formulations.) Finally, the surface charge integrals $q_j$~\eqref{eq:surf_charge} are small and decrease as $h\to 0$, as expected, even without the low-frequency correction ($\xi=0$).

We next examine the low-frequency regime on the same surface $\Gamma=\mathbb{S}^2$, now holding the mesh fixed at $h\approx 0.16$ across all experiments and sweeping the wavelength $\lambda=2\pi/k$ from $\lambda/d=1$ up to $\lambda/d=10^{16}$; results are reported in Table~\ref{tab:low_frequency}. As expected, the (R)D-MCFOIE formulations are largely unaffected by decreasing frequency, consistent with the absence of low-frequency breakdown in the magnetic equations. More surprisingly, the (R)D-ECFOIE formulations also remain well-behaved throughout: field errors, divergence errors, and surface charge integrals all stay at the level of the discretization error, and iteration counts remain stable. This is in stark contrast to their indirect counterparts~\cite{burbano2025maxwell} on the same problem. The stabilization parameter $\xi=\pi\cdot10^{4}$ has negligible effect on accuracy or conditioning here, consistent with the absence of breakdown. We attribute this robustness to the simple connectedness of the sphere, which, for these formulations, appears to suppress the near-kernel components responsible for low-frequency breakdown in more general, non-simply-connected geometries, as confirmed by the results on the torus and flower surface below.

Still on the sphere, Table~\ref{tab:higher_frequencies_sphere} reports results at higher frequencies, with meshes refined to maintain $h\approx\lambda/10$ so that the discretization error stays approximately constant across all $k$. All formulations achieve the expected level of accuracy for the solver parameters used, remain well-conditioned with no sign of high-frequency breakdown, and the regularized formulations consistently require fewer GMRES iterations than their unregularized counterparts.

To further investigate the low-frequency behavior of the (R)D-ECFOIE, we consider the same planewave incident field on two additional surfaces: a torus with major and minor radii $1$ and $1/2$, and the (non-axisymmetric) flower-shaped surface~\cite{wildman2004accurate}. These surfaces are parametrized by
\begin{equation}\begin{aligned}
\text{torus:}\quad & \left[(1+\tfrac{1}{2}\cos\theta)\cos\phi,\;(1+\tfrac{1}{2}\cos\theta)\sin\phi,\;\tfrac{1}{2}\sin\theta\right], \\
\text{flower:}\quad & \sqrt{0.8+0.5(\cos(2\phi)-1)(\cos(4\theta)-1)}\left[\cos(\phi)\sin(\theta),\sin(\phi)\sin(\theta),\cos(\theta)\right],
\end{aligned}
\label{eq:lf_surfaces}\end{equation}
for $(\theta,\phi)\in[0,2\pi]\times[0,2\pi]$ and $(\theta,\phi)\in[0,\pi]\times[0,2\pi]$, respectively, with diameters of approximately $d=3$ and $d\approx 3.35$, respectively. 

Table~\ref{tab:low_frequency_torus_pw} reports divergence errors, surface charge integrals, and GMRES iteration counts for the (R)D-ECFOIE formulations on both surfaces, for $\xi=0$ and $\xi=\pi\cdot10^4$, across the same low-frequency range as in Table~\ref{tab:low_frequency}. For the flower surface, a mild breakdown is observed: divergence errors remain small at all frequencies, but the surface charge integrals grow to be orders of magnitude larger, while GMRES iteration counts remain stable. Setting $\xi=\pi\cdot10^4$ effectively brings the surface charge below the level of the divergence error. For the torus, however, the breakdown is more pronounced when $\xi=0$: GMRES iteration counts grow as the frequency decreases and surface charge integrals remain large, indicating a severe loss of accuracy. The stabilization parameter $\xi=\pi\cdot10^4$ makes a decisive difference in this case, suppressing the surface charge and fully restoring accuracy---as jointly confirmed by the divergence error, the iteration counts, and the surface charge integrals. The contrast in behavior between the two surfaces is interesting: we speculate that the non-trivial topology of the torus (which has a non-simply-connected boundary and a non-trivial first homology group) may be responsible, as it introduces additional near-kernel components in the low-frequency limit that the stabilization term is designed to suppress.

\begin{table}[htb]
\centering
 \footnotesize
 \renewcommand{\arraystretch}{1.0}
\begin{tabular}{ccccccccccccc}
\toprule
 & \multicolumn{6}{c}{Torus} & \multicolumn{6}{c}{Flower} \\
\cmidrule(lr){2-7}\cmidrule(lr){8-13}
 & \multicolumn{3}{c}{D-ECFOIE; $\xi = 0$} & \multicolumn{3}{c}{D-ECFOIE; $\xi=\pi\cdot10^{4}$} & \multicolumn{3}{c}{D-ECFOIE; $\xi = 0$} & \multicolumn{3}{c}{D-ECFOIE; $\xi=\pi\cdot10^{4}$} \\
\cmidrule(lr){2-4}\cmidrule(lr){5-7}\cmidrule(lr){8-10}\cmidrule(lr){11-13}
$\lambda/d$ & $e_{\dive\! E^s}$ & $|q_1|$ & \#iter & $e_{\dive\! E^s}$ & $|q_1|$ & \#iter & $e_{\dive\! E^s}$ & $|q_1|$ & \#iter & $e_{\dive\! E^s}$ & $|q_1|$ & \#iter \\ \midrule
$10^{16}$ & 8.24\text{e-}4 & {\bf 1.65\text{e-}0} &  227 & 7.08\text{e-}4 & 2.72\text{e-}6 &  178 & 7.55\text{e-}5 & 2.98\text{e-}3 &  75 & 7.58\text{e-}5 & 6.25\text{e-}9 &  76 \\
$10^{08}$ & 8.24\text{e-}4 & {\bf 1.65\text{e-}0} &  227 & 7.08\text{e-}4 & 2.72\text{e-}6 &  178 & 7.55\text{e-}5 & 2.98\text{e-}3 &  75 & 7.58\text{e-}5 & 6.25\text{e-}9 &  76 \\
$10^{04}$ & 8.24\text{e-}4 & {\bf 1.65\text{e-}0} &  227 & 7.08\text{e-}4 & 2.74\text{e-}6 &  179 & 7.56\text{e-}5 & 3.01\text{e-}3 &  76 & 7.61\text{e-}5 & 6.31\text{e-}9 &  77 \\
$10^{02}$ & 7.88\text{e-}4 & {\bf 2.23\text{e-}1} &  151 & 7.09\text{e-}4 & 3.01\text{e-}7 &   89 & 7.48\text{e-}5 & 2.86\text{e-}3 &  78 & 7.56\text{e-}5 & 3.72\text{e-}9 &  78 \\
$10^{01}$ & 7.08\text{e-}4 & 2.52\text{e-}3 &  82 & 6.67\text{e-}4 & 4.08\text{e-}9 &  78 & 7.80\text{e-}5 & 6.89\text{e-}5 &  84 & 7.43\text{e-}5 & 1.03\text{e-}10 &  80 \\
$10^{00}$ & 3.43\text{e-}4 & 6.96\text{e-}6 &  89 & 3.41\text{e-}4 & 1.12\text{e-}11 &  90 & 2.70\text{e-}5 & 5.83\text{e-}7 & 110 & 2.69\text{e-}5 & 7.28\text{e-}13 & 111 \\
\midrule\addlinespace[3pt]
 & \multicolumn{6}{c}{Torus} & \multicolumn{6}{c}{Flower} \\
\cmidrule(lr){2-7}\cmidrule(lr){8-13}
& \multicolumn{3}{c}{RD-ECFOIE; $\xi = 0$} & \multicolumn{3}{c}{RD-ECFOIE; $\xi=\pi\cdot10^{4}$} & \multicolumn{3}{c}{RD-ECFOIE; $\xi = 0$} & \multicolumn{3}{c}{RD-ECFOIE; $\xi=\pi\cdot10^{4}$} \\
\cmidrule(lr){2-4}\cmidrule(lr){5-7}\cmidrule(lr){8-10}\cmidrule(lr){11-13}
$\lambda/d$ & $e_{\dive\! E^s}$ & $|q_1|$ & \#iter & $e_{\dive\! E^s}$ & $|q_1|$ & \#iter & $e_{\dive\! E^s}$ & $|q_1|$ & \#iter & $e_{\dive\! E^s}$ & $|q_1|$ & \#iter \\ \midrule
$10^{16}$ & 8.20\text{e-}4 & {\bf 1.60\text{e-}0} &  211 & 7.08\text{e-}4 & 3.09\text{e-}6 &  53 & 7.59\text{e-}5 & 4.19\text{e-}3 &  58 & 7.60\text{e-}5 & 2.64\text{e-}8 &  60 \\
$10^{08}$ & 8.20\text{e-}4 & {\bf 1.60\text{e-}0} &  211 & 7.08\text{e-}4 & 3.09\text{e-}6 &  53 & 7.59\text{e-}5 & 4.19\text{e-}3 &  58 & 7.60\text{e-}5 & 2.64\text{e-}8 &  60 \\
$10^{04}$ & 8.20\text{e-}4 & {\bf 1.60\text{e-}0} &  211 & 7.08\text{e-}4 & 3.10\text{e-}6 &  54 & 7.59\text{e-}5 & 4.21\text{e-}3 &  59 & 7.61\text{e-}5 & 2.66\text{e-}8 &  61 \\
$10^{02}$ & 7.85\text{e-}4 & {\bf 2.15\text{e-}1} &  90 & 7.09\text{e-}4 & 3.63\text{e-}7 &  49 & 7.53\text{e-}5 & 4.33\text{e-}3 &  60 & 7.55\text{e-}5 & 6.36\text{e-}9 &  82 \\
$10^{01}$ & 7.08\text{e-}4 & 2.52\text{e-}3 &  56 & 6.67\text{e-}4 & 4.08\text{e-}9 &  56 & 7.78\text{e-}5 & 7.00\text{e-}5 &  65 & 7.37\text{e-}5 & 1.06\text{e-}10 &  64 \\
$10^{00}$ & 3.43\text{e-}4 & 6.97\text{e-}6 & 100 & 3.41\text{e-}4 & 1.12\text{e-}11 & 100 & 2.81\text{e-}5 & 3.66\text{e-}7 & 109 & 2.81\text{e-}5 & 5.57\text{e-}13 & 109 \\
\bottomrule
\end{tabular}
\caption{Relative errors, induced surface charge~\eqref{eq:surf_charge}, and GMRES iteration counts (with a relative tolerance of $10^{-6}$) for the (R)D-ECFOIE formulations applied to the electromagnetic scattering problem by a torus and a flower surface under planewave illumination~\eqref{eq:planewave} at low frequencies. The results are obtained using fixed surface meshes of size $h \approx 0.17$. In all cases, the parameter $\eta = 100\pi$ is used, and two different values of the parameter $\xi$---introduced to mitigate the low-frequency breakdown---are shown, demonstrating its effectiveness in reducing the surface charge values. Bold entries indicate cases severely affected by low-frequency breakdown. 
}
\label{tab:low_frequency_torus_pw}
\end{table}

To further illustrate the BIE formulations for disconnected scatterers, our final example considers planewave scattering~\eqref{eq:planewave}, with $p=(1,0,0)$ and $d=2^{-\frac12}(0,1,-1)$, by two configurations of disjoint tori forming disconnected scatterers with boundaries $\Gamma=\Gamma_1\cup\Gamma_2$. The two configurations, displayed in Figure~\ref{fig:meshes_tori}, consist of a pair of interlocking nearly touching tori (left) and a pair of adjacent tori (right); each torus has minor radius $1/2.1$ rather than $1/2$, a slight reduction chosen to create a narrow gap between the interlocking pair. All four BIE formulations are solved on the fixed meshes shown in that figure, with $h\approx0.14$ ($N=54648$), at three frequencies ($k/\pi=10^{-8}$, $10^{-1}$, $10^{0}$) and for $\xi=0$ and $\xi=\pi\cdot10^4$; results are reported in Table~\ref{tab:two_tori_pw} using a GMRES tolerance of $10^{-4}$, matched to the expected discretization accuracy. At low frequency ($k=10^{-8}\pi$), low-frequency breakdown is visible in the (R)D-ECFOIE with $\xi=0$: the surface charge integrals, here measured via $\max\{|q_1|,|q_2|\}$, are one to two orders of magnitude larger than the divergence errors, most noticeably for the adjacent configuration. Setting $\xi=\pi\cdot10^4$ reduces the surface charge to $\sim10^{-7}$ across both configurations, restoring accuracy. At $k=10^{-1}\pi$ and $k=\pi$, all formulations are well-behaved with no signs of breakdown, and accuracy levels are comparable across both configurations. Figure~\ref{fig:electric_interlocking_tori_fields} shows the real parts of the $x$- and $z$-components of the total electric and magnetic fields, $\mathrm{Re}(E_x)$ and $\mathrm{Re}(H_z)$, for the interlocking configuration at $k=\pi$, confirming that GP-DIM resolves the fields accurately in the challenging nearly touching geometry.

\begin{table}[htb]
\centering
\footnotesize
\renewcommand{\arraystretch}{1.0}
\begin{tabular}{llcccccccc}
\toprule
 & & \multicolumn{3}{c}{D-ECFOIE~\eqref{eq:electric_BIE}; $\xi=0$} & \multicolumn{3}{c}{D-ECFOIE~\eqref{eq:mod_DECFOIE}; $\xi=\pi\cdot10^{4}$} & \multicolumn{2}{c}{D-MCFOIE~\eqref{eq:magnetic_BIE}} \\
\cmidrule(lr){3-5}\cmidrule(lr){6-8}\cmidrule(lr){9-10}
Tori & $k/\pi$ & $e_{\dive E^s}$ & $\max\{|q_1|,|q_2|\}$ & iter. & $e_{\dive E^s}$ & $\max\{|q_1|,|q_2|\}$ & iter. & $e_{\dive H^s}$ & iter.\\ \midrule
\multirow{3}{*}{Interlocking} & $10^{-8}$ & 5.69e-04 & 2.42e-03 & 49 & 4.68e-04 & 6.07e-07 & 52 & 1.99e-04 & 58 \\
& $10^{-1}$ & 3.86e-04 & 1.43e-03 & 58 & 3.66e-04 & 2.26e-06 & 55 & 1.11e-03 & 119 \\
& $10^{-0}$ & 8.65e-04 & 1.31e-04 & 127 & 5.39e-04 & 3.50e-07 & 126 & 4.31e-03 & 69 \\[2pt]
\multirow{3}{*}{Adjacent}     & $10^{-8}$ & 5.12e-04 & 2.20e-02 & 19 & 5.84e-04 & 1.41e-07 & 20 & 2.44e-03 & 24 \\
& $10^{-1}$ & 7.79e-04 & 1.42e-02 & 26 & 3.39e-04 & 1.34e-07 & 23 & 5.66e-04 & 47 \\
& $10^{-0}$ & 8.11e-04 & 5.91e-05 & 63 & 8.51e-04 & 1.27e-07 & 63 & 8.39e-04 & 47 \\
\midrule\addlinespace[3pt]
 & & \multicolumn{3}{c}{RD-ECFOIE~\eqref{eq:refor}; $\xi=0$} & \multicolumn{3}{c}{RD-ECFOIE~\eqref{eq:mod_reg_DECFOIE}; $\xi=\pi\cdot10^{4}$} & \multicolumn{2}{c}{RD-MCFOIE~\eqref{eq:reg_mag_BIE}} \\
\cmidrule(lr){3-5}\cmidrule(lr){6-8}\cmidrule(lr){9-10}
Tori & $k/\pi$ & $e_{\dive E^s}$ & $\max\{|q_1|,|q_2|\}$ & iter. & $e_{\dive E^s}$ & $\max\{|q_1|,|q_2|\}$ & iter. & $e_{\dive H^s}$ & iter.\\ \midrule
\multirow{3}{*}{Interlocking} & $10^{-8}$ & 4.60e-04 & 2.41e-02 & 33 & 4.71e-04 & 1.35e-07 & 40 & 3.88e-04 & 40 \\
& $10^{-1}$ & 3.23e-04 & 1.05e-03 & 42 & 2.30e-04 & 5.07e-07 & 46 & 2.56e-04 & 51 \\
& $10^{-0}$ & 4.78e-04 & 2.46e-03 & 129 & 5.37e-04 & 5.12e-07 & 127 & 1.14e-03 & 58 \\[2pt]
\multirow{3}{*}{Adjacent}     & $10^{-8}$ & 3.09e-04 & 2.16e-02 & 12 & 2.99e-04 & 1.27e-07 & 14 & 1.45e-03 & 20 \\
& $10^{-1}$ & 3.23e-04 & 5.75e-03 & 19 & 2.40e-04 & 6.01e-08 & 19 & 1.41e-04 & 36 \\
& $10^{-0}$ & 6.44e-04 & 2.74e-05 & 66 & 6.13e-04 & 9.71e-08 & 65 & 4.84e-04 & 39 \\
\bottomrule
\end{tabular}
\caption{
Divergence errors, surface charge integrals, and GMRES iteration counts (relative tolerance $10^{-4}$) for all four BIE formulations applied to planewave scattering~\eqref{eq:planewave} by two configurations of non-simply-connected scatterers at a low ($k=10^{-8}\pi$) and a moderate ($k=10^{-1}\pi$) frequency. The top half shows the direct formulations and the bottom half their analytically preconditioned counterparts; two values of $\xi$ are shown for the (R)D-ECFOIE. A fixed mesh with $h\approx0.14$ ($N=54648$) is used, with $\eta=100\pi$ throughout.}
\label{tab:two_tori_pw}
\end{table}

 \begin{figure}[htb]
   \centering
   \includegraphics[width=0.45\textwidth]{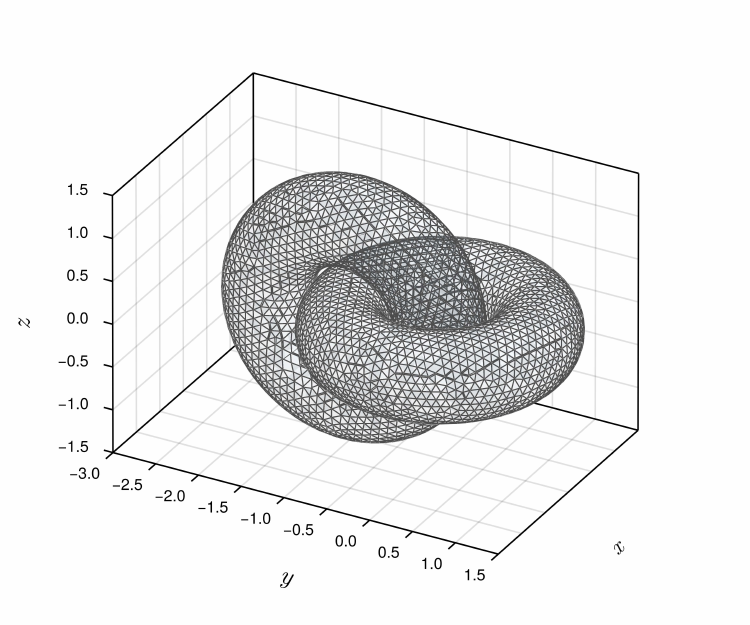}
   \includegraphics[width=0.45\textwidth]{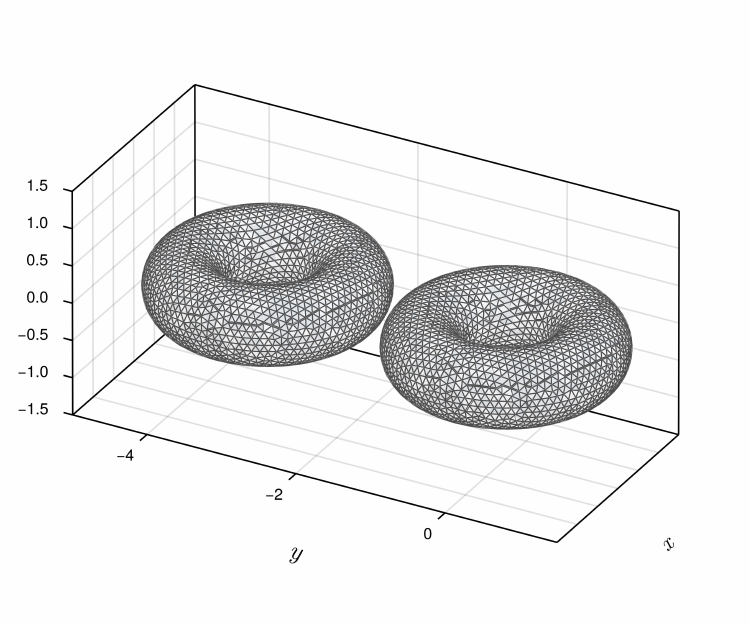}
   \caption{Surface meshes for the two non-simply-connected scatterer configurations used in Table~\ref{tab:two_tori_pw}: interlocking tori (left) and adjacent tori (right), both with $h\approx0.14$ ($N=54648$).}
   \label{fig:meshes_tori}
 \end{figure}

 \begin{figure}[htb]
   \centering
   \includegraphics[width=0.45\textwidth]{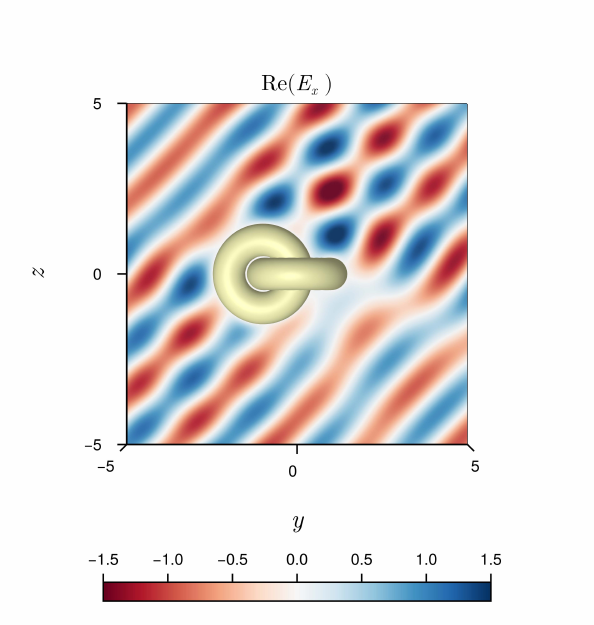}
   \includegraphics[width=0.45\textwidth]{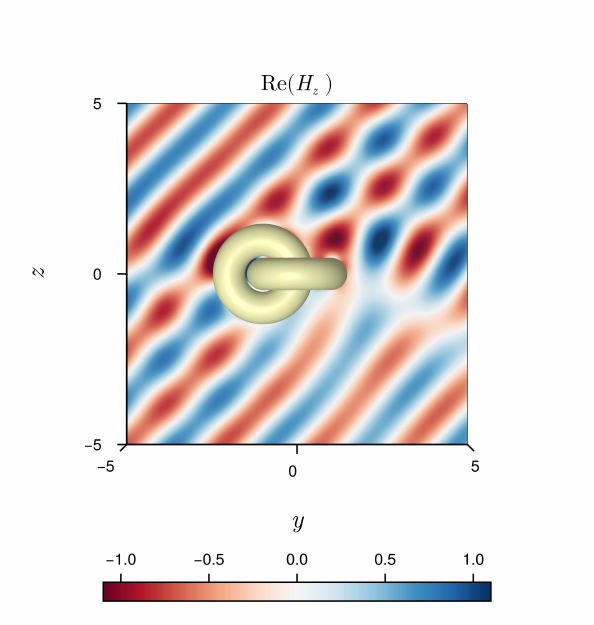}
   \caption{Real part of the $x$-component of the total electric field $\mathrm{Re}(E_x)$ (left) and the $z$-component of the total magnetic field $\mathrm{Re}(H_z)$ (right), computed for the interlocking tori configuration at $k=\pi$, with $\eta=100\pi$ and $\xi=\pi\cdot10^4$.}
   \label{fig:electric_interlocking_tori_fields}
 \end{figure}

\section{Conclusions}
We have developed direct combined-field-only boundary integral formulations---D-ECFOIE and D-MCFOIE---for time-harmonic electromagnetic scattering by smooth PEC obstacles, building on the \emph{Maxwell à la Helmholtz} framework of~\cite{burbano2025maxwell}. Unlike their indirect counterparts, the proposed formulations take as unknowns the physical field traces on the scattering surface; in particular, the D-MCFOIE yields the surface electric currents directly. We proved unique solvability at all frequencies $k>0$, introduced Calderón regularizations that render the equations of Fredholm second kind, and resolved the low-frequency breakdown of the electric-field formulation via a charge-conservation-enforcing modification. Numerical experiments with a GP-DIM-based Nyström solver implemented in \texttt{Inti.jl} confirmed the accuracy and robustness of all formulations across a range of geometries and frequencies. Future work includes extensions to Sobolev-space settings and transmission problems, all within the same Helmholtz-operator framework.

\bibliographystyle{abbrv}
\bibliography{References}

\end{document}